%% file: 2016_Viscosityjumps.tex
\documentclass[10pt]{article}
\usepackage{amssymb,amsmath}
\usepackage{amsthm}
\usepackage{amsmath}
  \usepackage{graphicx}
\usepackage[colorlinks=false]{hyperref}
\graphicspath{{/}}
\usepackage{tikz,pgfplots,pgfplotstable} 
\usepackage{mathtools}

\title{Highly sparse surface couplings for subdomain-wise isoviscous Stokes finite element discretizations}

\newtheorem{theorem}{Theorem}
\newtheorem{lemma}{Lemma}
\newtheorem{remark}{Remark}

\def\bs#1{\boldsymbol{#1}}
\def\bu{\bs{u}}
\def\bv{\bs{v}}
\def\bn{\bs{n}}

\def\jump#1{\ensuremath{[\![#1]\!]}}

\def\div{\mathop{\rm div}}
\def\trace{\mathop{\rm tr}}
\def\dev{\mathop{\rm dev}}
\def\sym{\mathop{\rm sym}}
\def\tr{\mathop{\rm tr}}

\def\Id{\mathop{\rm Id}}

\begin{document}
\author{Markus~Huber$^{1}$
\and Ulrich~R\"ude$^{2}$
\and Christian~Waluga$^{3}$ 
\and Barbara~Wohlmuth$^{1}$}
\footnotetext[1]{Institute for Numerical Mathematics (M2), Technische Universit{\"a}t M{\"u}nchen, Boltzmannstrasse~3, D--85748 Garching b. M\"unchen, Germany}
\footnotetext[2]{Department of Computer Science 10, FAU Erlangen-N{\"u}rnberg, Cauerstra\ss{}e 6, D--91058 Erlangen, Germany}
\footnotetext[3]{liNear GmbH, Kackertstr. 11, D--52072 Aachen, Germany; corresponding author}

\maketitle
\begin{abstract}
The Stokes system with constant viscosity can be cast into different formulations by exploiting the incompressibility constraint.
For instance 
the strain in the weak formulation can  be replaced by the gradient 
to decouple 
the  velocity components in the different coordinate directions. 
Thus the discretization of the simplified problem
leads to
fewer nonzero entries in the stiffness matrix.
This is of particular  interest in large scale simulations where
a reduced memory bandwidth requirement
can help to significantly accelerate the computations.
In the case of a piecewise constant viscosity, as it typically arises in multi-phase flows,
or when the boundary conditions involve traction, the situation is more complex,
and one has to treat the cross derivatives in the original
Stokes system with care.
A naive application of the standard vectorial Laplacian results in a
physically incorrect solution, while formulations based on the strain increase the computational effort everywhere, even when the inconsistencies arise only from an 
incorrect treatment in a small fraction of the computational domain.
Here we propose a new approach
that is consistent with the strain-based formulation and preserves the decoupling advantages of the gradient-based formulation in isoviscous subdomains. 
The modification
is equivalent to locally changing the discretization stencils, hence the more expensive discretization is restricted to a lower dimensional interface,
making the additional computational cost asymptotically negligible.
We demonstrate the consistency and convergence properties of the method and show that in a massively parallel setup, the multigrid solution of the resulting discrete systems is faster than for the classical strain-based formulation.
Moreover, we give an application example which is inspired by geophysical research.
\end{abstract}

\section{Introduction}
\label{sec:intro}
\input{introduction}

\section{A symmetric decoupled weak formulation}
\label{sec:weak-formulation}
\input{decoupled}
\section{Stability and uniform coercivity in the discrete setting}
\label{sec:stability}
\input{stability}
\section{Numerical examples}
\label{sec:numerics}
\input{numerics}

\section{Massively parallel multigrid}
\input{large}
\section{Geophysical example}
\input{geo}
\subsubsection*{Acknowledgements}
This work was supported (in part) by the German Research Foundation (DFG) through the Priority Programme 1648 "Software for Exascale Computing" (SPP\-EXA). The authors gratefully acknowledge the Gauss Centre for Supercomputing (GCS) for providing computing time through the John von Neumann Institute for Computing (NIC) on the GCS share of the supercomputer JUQUEEN at J\"ulich Supercomputing Centre (JSC).
\bibliographystyle{abbrv}
\bibliography{literature}
\end{document}

%% file: introduction.tex
The non-isoviscous Stokes equations are an important mathematical model for a wide range of applications and serve as a building block for more complicated fluid models.
In several applications of interest, however,
the model can be considerably simplified by formally decomposing the simulation domain into isoviscous sub-domains. To this end, consider the incompressible Stokes equations
\begin{subequations}
\label{eq:stokes-strong}
\begin{align}
\label{eq:stokes-strong-momentum}
-\div\bs{\sigma}(\bu,p) &= \bs{f},\\
\label{eq:stokes-strong-mass}
\div\bu &= 0,
\end{align}
\end{subequations}
in a polyhedral domain $\Omega \subset \mathbb{R}^d$, $d=2,3$, where $\bu=[u_1,\dots,u_d]^\top$ denotes the velocity, $p$ the pressure, and $\bs{f} = [f_1,\dots,f_d]^\top]$ is an external forcing term. The stress tensor is defined as
\begin{equation}
\label{eq:stress-sym-nabla}
\bs{\sigma}(\bu,p) := 2\mu \sym\nabla\bu - p\cdot\Id,
\end{equation}
where $\sym T := \frac12 (T + T^\top)$ denotes the symmetric part of a
tensor $T \in \mathbb{R}^{d \times d}$, and $\mu$ is a positive scalar
viscosity field. The focus of this work lies on a model setting
resulting from the assumption that the viscosity is a piecewise
constant function with a possibly large contrast, i.e.,
\begin{equation}
\mu = \begin{cases}
\mu_1              &\text{in}~\Omega_1,\\
\mu_2  &\text{in}~\Omega_2
\end{cases}
\end{equation}
with $\mu_1 \geq \mu_2 > 0$.

Such models have relevant applications, for instance in incompressible two-phase flow simulations (cf. \cite{gross2011numerical} and the references therein), or in fundamental geophysical studies; cf. e.g. \cite{HC89,gerya2009introduction,weismueller_15}.

To simplify the exposition, the disjoint subdomains $\Omega_1$ and $\Omega_2$ are assumed to be polyhedral such that $\overline{\Omega} = \overline{\Omega}_1 \cup \overline{\Omega}_2$, and the interface between both subdomains is denoted by $\Gamma_{12} := \partial\Omega_1 \cap \partial\Omega_2$. The equations \eqref{eq:stokes-strong} are then equivalent to the interface formulation
\begin{subequations}
\label{eq:stokes-strong-interface}
\begin{align}
\label{eq:stokes-strong-interface-momentum}
-\mu_i\Delta\bu + \nabla p &= \bs{f}, &&\text{in}~\Omega_i,~i=1,2\\
\label{eq:stokes-strong-interface-mass}
\div\bu &= 0, &&\text{in}~\Omega_i,\\
\label{eq:stokes-strong-interface-continuity}
\jump{\bu} = 0,\quad\jump{\bs{\sigma}(\bu,p)\cdot\bn} & = 0, && \text{on}~\Gamma_{12},\end{align}
\end{subequations}

\noindent which is discussed e.g. in
\cite{QuaVal:1999,ito2006interface,olshanskii2006analysis}. Here
$\bn_i$ stands for the outward pointing unit normal on
$\partial\Omega_i$, and we omit the index where there is no
ambiguity. Moreover, $\jump{v} := v|_{\Gamma_{12}\cap\Omega_1} -
v|_{\Gamma_{12}\cap\Omega_2} $ denotes the oriented jump of a scalar
or vectorial quantity $v$ across the interface, and  $\jump{ T \cdot
  \bn } := T|_{\Gamma_{12}\cap\Omega_1} \cdot \bn_1 +
T|_{\Gamma_{12}\cap\Omega_2} \cdot \bn_2$ stands for the jump of the normal component of a
tensorial quantity $T$.

In terms of boundary conditions, we complement the equations \eqref{eq:stokes-strong-interface} by \emph{homogeneous Dirichlet} conditions $\bu = \mathbf{0}$ on $\Gamma_D \subset \partial \Omega$ and prescribe either \emph{traction-free} boundary conditions
\begin{subequations}
\label{eq:boundary}
\begin{align}
\label{eq:stokes-strong-boundary-tractionfree}
\bs{\sigma}(\bu,p)\cdot\bn & = \mathbf{0}, \qquad \text{on}~\Gamma_{F} \subset \partial \Omega,
\intertext{or \emph{free-slip} boundary conditions}
\label{eq:stokes-strong-boundary-freeslip}
{\bu \cdot \bn} = 0,\quad \bn \times \bs{\sigma}(\bu,p)\cdot\bn & = 0, \qquad \text{on}~\Gamma_{F} \subset \partial \Omega.
\end{align}
\end{subequations}
 As it is standard, we
assume $\Gamma_D \cap \Gamma_F = \emptyset$, $\bar \Gamma_D \cup \bar
\Gamma_F = \partial \Omega$, and the $(d-1)$-dimensional measure of
$\Gamma_D$ is positive such that Korn-type estimates hold. 

The important difference between the single domain  \eqref{eq:stokes-strong} and the multi-domain formulation \eqref{eq:stokes-strong-interface} is that the presence of cross-derivatives in \eqref{eq:stokes-strong-interface} is limited to $\Gamma_{12} $.  At the interface $\Gamma_{12}$, the continuity of the normal stresses is enforced via \eqref{eq:stokes-strong-interface-continuity}. Moreover at the boundary part $\Gamma_F$, the traction is (partly) constrained via 
\eqref{eq:boundary}. Thus in the following, both $\Gamma_F $ and $\Gamma_{12}$ will be called \emph{interface} to simplify notation. The equivalence of \eqref{eq:stokes-strong} and  \eqref{eq:stokes-strong-interface} results from the fact that for solenoidal velocities, we have
\begin{equation}
2\div(\sym\nabla\bu) = \Delta\bu + \nabla \div \bu = \Delta\bu,
\end{equation}
which decouples the velocity components in the momentum equation.
Preserving this decoupling property in a discrete sense is of special interest in large-scale computations since the number of nodes at the interface increases with ${\mathcal O} (h^{-d+1})$ while the total number of degrees of freedom grows with  ${\mathcal O} (h^{-d})$. Thus, asymptotically, the number of nodes
associated with the interior of the subdomains dominates, and the computational cost
can be significantly reduced by using an approach 
that decouples in the interior of the isoviscous subdomains. This observation 
does not only hold for traditional element-wise assembling procedures but also for the matrix-vector product with the stiffness matrix. A fast realization of these basic operations is 
essential to obtain a fast time-to-solution of many iterative solvers such as multigrid methods. It is however not obvious how to exploit the previously discussed decoupling property for large-scale applications.  This fact is related to the observation that interface problems often result in a saddle point 
structure that may give rise to stability concerns.

To elaborate on this, let us choose a natural setting for the
interface problem by defining appropriate function spaces. For the
velocity space we choose $\bs{V} := \{\bv \in H^1 (\Omega)^d~:~
\bv = \mathbf{0} \text{ on } \Gamma_D , \int_{\Gamma_F} \bv \cdot \bn \ ds =0 \}$ and  $\bs{V} := \{\bv \in H^1 (\Omega)^d~:~ \bv
= \mathbf{0} \text{ on } \Gamma_D, \bv \cdot \bn = 0 \text{ on }
\Gamma_F\}$ for \eqref{eq:stokes-strong-boundary-tractionfree}
and \eqref{eq:stokes-strong-boundary-freeslip}, respectively.
We point out that the definition of the velocity space automatically guarantees that
$\int_{\partial \Omega} \bv \cdot \bn \ ds = \int_\Omega \div \bv \ dx =0$, $\bv \in \bs{V} $ and thus for all our settings, we only work with
$Q := \{~q \in L^2(\Omega)~:~\int_\Omega \mu^{-1} q\,{\rm d}x =
0~\}$ as test and solution space for the pressure. This choice  ensures
 uniqueness of the
pressure with respect to the constant mode. 
 The weak problem will then read as: given
$\bs{f} \in \bs{V}'$, find $(\bu,p) \in \bs{V} \times Q$, such that
\begin{subequations}
\label{eq:weaks}
\begin{align} \label{eq:weakmo}
a(\bu,\bv) + b(\bv,p) &= \langle \bs{f},\bv \rangle, && \text{for all}~ \bv \in \bs{V},\\  \label{eq:weakma}
b(\bu,q) &= 0, && \text{for all}~ q \in Q,
\end{align}
\end{subequations}
where $b(\bv,q) := -(\div\bv,q)_0$ is the standard weak form of the divergence and $(\cdot,\cdot)_0$ denotes the $L^2$ scalar product on $\Omega$. The exact form of $a(\cdot,\cdot)$ however depends on the strong formulation on
 which we base our derivation. Starting from \eqref{eq:stokes-strong} leads to the bilinear form
\begin{align} \label{eq:standard}
a_{\sym}(\bu,\bv) := (2\mu\sym\nabla\bu,\sym\nabla\bv) .
\end{align}
By Korn's second inequality, we obtain positive definiteness of $a_{\rm sym}(\cdot,\cdot)$ on $\bs{V}\times\bs{V}$. The drawback of this weak form is that it does not exploit the knowledge that the velocity components are decoupled in the strong momentum balance inside the subdomains, which leads to a reduced sparsity in the finite-dimensional setting.

Another straightforward approach would be to start from the interface formulation \eqref{eq:stokes-strong-interface}. Integrating by parts over each subdomain and incorporating the momentum balance, we obtain a weak Laplacian which is augmented with consistency terms at the interface
\begin{align}
\label{eq:aint}
a_{\rm int}(\bu,\bv) := (\mu\nabla\bu,\nabla\bv) + \sum_{i=1}^2(\mu\nabla\bu^\top\bn,\bv)_{\partial\Omega_i\cap\Gamma_{12}}+ (\mu\nabla\bu^\top\bn,\bv)_{\Gamma_{F}},
\end{align}

\noindent cf. e.g. \cite{LimIdeRosOna:2007}, where a similar construction for the special case of boundary conditions involving the surface traction is considered. The above formulation, while being consistent and decoupled inside the subdomains, has two major drawbacks: Firstly, due to the additional consistency terms, this form sacrifices symmetry. Secondly, it is --
to the authors' best knowledge -- not possible to prove coercivity of
$a_{\rm int}(\cdot,\cdot)$ for large viscosity contrasts. Let us at this point underline that simply neglecting the consistency terms in a discrete solver will, in general, result in convergence to a wrong solution if $\mu_1 \neq \mu_2$ or $\Gamma_F\ne\emptyset$.
Limache~et~al.~\cite{LimIdeRosOna:2007} interpret the lack of consistency as a violation of the principle of objectivity, a basic axiom of continuum mechanics.

In this article, we introduce 
a framework for incompressible flow with subdomain-wise constant viscosities, that exploits the increased sparsity of the vectorial Laplacian and thus -- assuming comparable convergence rates -- allows 
for a significantly faster iterative solution as compared to the classical strain based formulation.
Our goal is to provide a mathematical sound reformulation
in a variationally consistent setting and to report on 
the resulting performance and scalability results. The paper is organized as follows:
In the following Section~\ref{sec:weak-formulation}, we shall derive an
alternative bilinear form which is both consistent and symmetric,
therefore eliminates the major drawback of the form $a_{\rm int}
(\cdot,\cdot)$ defined above. Subsequently, in
Section~\ref{sec:stability} we discuss the circumstances under which
the modified form leads us to a stable and convergent approximation
scheme in case of finite dimensional velocity and pressure
spaces. Particularly, we present and analyze a lowest-order stabilized
method which has a decoupling property built-in and suppresses
spurious pressure modes. In Section~\ref{sec:numerics}, we give some illustrative numerical examples to complement our theory and to report on the discretization error for the different formulations. Section~\ref{sec:large} is devoted to performance considerations. More precisely, we consider the influence of the formulation on the  time to solution for a highly scalable and parallel all-at-once multigrid solver. Finally in Section~\ref{sec:geo}, we present a simplified geophysical model example and show that the new modified scheme yields results which are in agreement with the physically accepted strain-based formulation of the Stokes problem.

%% file: decoupled.tex
In this section, we provide two different weak formulations 
that decouple in the subdomains and involve cross derivative terms only across the interface.
The approach proposed in this work is primarily based on an observation 
that can be traced back at least to Pironneau~\cite[Sect.\,6.4]{pironneau1989finite}, namely, that  we can rewrite
\begin{equation}
\label{eq:pironneau-identity}
(2\sym\nabla\bu,\sym\nabla\bv)_{\Omega_i} - (\div\bu,\div\bv)_{\Omega_i}
  = (\nabla\bu,\nabla\bv)_{\Omega_i} + (\nabla\bu^\top\bn - \div\bu\cdot\bn,\bv)_{\partial\Omega_i}.
\end{equation}

\noindent For a formal proof and suitable assumptions, we refer to Lemma~\ref{lem:equivalence-jump} below. Using \eqref{eq:pironneau-identity}, we can easily derive a symmetric bilinear form 
\begin{align} \label{eq:newbil}
a_{\tr}(\bv,\bu) &:= a_{\sym} (\bv,\bu)
 - ( \mu \div\bu,\div\bv) \\ \nonumber
&= a_{\dev}(\bu,\bv) - (\mu(1-\tfrac2d)\div\bu,\div\bv),
\end{align}

\noindent where $a_{\dev}(\bu,\bv) := (2 \mu \dev \sym\nabla\bu,
\dev\sym\nabla\bv) $.  Here $\dev T := T - \tfrac{1}{d}
\trace(T)\cdot\Id$ denotes the deviatoric part of a tensor $T$. Thus, we directly obtain a decoupling property for the momentum part also in the weak setting. The above modification amounts to formally changing the stress tensor to
\begin{equation}
\label{eq:stress-dev-sym-nabla}
\begin{aligned}
\bs{\sigma}(\bu,p)
&= 2\mu \sym\nabla\bu - (p + \mu\div\bu)\cdot\Id \\
&= 2\mu \dev\sym\nabla\bu - (p + \mu(1-\tfrac2d)\div\bu)\cdot\Id .
\end{aligned}
\end{equation}

 It is obvious that $a_{\tr}(\cdot,\cdot)$ is symmetric, and moreover for
$d=2$ it is also coercive on $\bs{V} \times \bs{V}$ if $\Gamma_D
= \partial \Omega$ with respect to the standard $H^1$-norm, see, e.g.,
\cite[Lemma 4.2]{neff2015poincare}. For solenoidal elements of $\bs{V}$ the bilinear form \eqref{eq:newbil} is equal to \eqref{eq:standard}, and thus the modified problem is strongly consistent with the original Stokes problem \eqref{eq:stokes-strong} due to the incompressibility constraint $\div\bu=0$.

Before we discuss finite element approximations for the corresponding weak problem, we shall discuss two possible equivalent reformulations which can be used for the implementation to exploit the advantageous decoupling property. 

\subsection{Equivalence with a tangential interface coupling}
Before we can state our main results, we need to discuss some preliminary algebraic identities. These identities provide a useful tool to expand a trace-free tensor in terms of a rotational basis. For the two-dimensional setting, $d=2$, we define the orthogonal matrix
\begin{equation*}
R_2^1=\left(\begin{matrix}0&-1\\1&0\end{matrix}\right)
\end{equation*}

\noindent and set $n_d = n_2 = 1$. For the three-dimensional case, $d=3$, we accordingly define
\begin{equation*}
R_3^1 = \left(\begin{matrix}
	0 &  0 &  0 \\
    0 &  0 & -1 \\
    0 &  1 &  0
\end{matrix}\right),
\qquad
R_3^2 = \left(\begin{matrix}
	0 &  0 &  1 \\
    0 &  0 &  0 \\
   -1 &  0 &  0
\end{matrix}\right),
\qquad
R_3^3 = \left(\begin{matrix}
	0 & -1 &  0 \\
    1 &  0 &  0 \\
    0 &  0 &  0
\end{matrix}\right)
\end{equation*}
and set $n_d = n_3 = 3$. For $\mathbf{x} \in \mathbb{R}^2$ we have $R_2^1\mathbf{x} \perp \mathbf{x}$ and for $\mathbf{x} \in \mathbb{R}^3$ we get $R_3^j \mathbf{x} = \mathbf{e}_j \times \mathbf{x}$, where $\{ \mathbf{e}_j, 1 \le j \le 3 \}$, denotes the canonical basis of $\mathbb{R}^3$.

\begin{lemma}
\label{lem:rotation-identity}
Let $T \in \mathbb{R}^{d \times d}$, $d=2,3$. Then we have the equalities
\begin{align}
\label{eq:decomposition}
T - \mathop{\rm tr}\,(T) \cdot \Id &=  \sum_{j=1}^{n_d} R_d^j T^\top R_d^j, \\
\label{eq:traceless-rotation}
\dev{T} & = \sum_{j=1}^{n_d} R_d^j \dev{T}^\top R_d^j. 
\end{align}
\end{lemma}
\begin{proof} Both equalities are easily obtained from a straightforward computation.
For $d=3$, it is helpful to note that $\sum_{j=1}^{3} R_d^j R_d^j = -2 \Id$, and
\begin{align*}
\dev{T} & = T - \mathop{\rm tr}\,(T) \cdot \Id + \frac{2}{3} \mathop{\rm tr}\,(T) \cdot \Id= \sum_{j=1}^3 R_d^j  T^\top  R_d^j 
- \frac{1}{3} \mathop{\rm tr}\,(T)  \sum_{j=1}^3 R_d^j  R_d^j 
\\
&  =
\sum_{j=1}^3 R_d^j ( {T}^\top - \frac 13 \mathop{\rm tr}\,(T^\top) \cdot \Id) R_d^j 
 =
\sum_{j=1}^3 R_d^j \dev{T}^\top  R_d^j,
\end{align*}
thus \eqref{eq:traceless-rotation} follows from \eqref{eq:decomposition} 
and the definition of the deviatoric operator.
\end{proof}

The previous identities are essential for our derivation as we will show next. Let
${\mathcal T}_h (\Omega_i)$ be a family of shape regular triangulations on $\Omega_i$ such that each vertex on $\Gamma_{12}$ is a vertex of both ${\mathcal T}_h (\Omega_1)$ and ${\mathcal T}_h (\Omega_2)$, with at least one inner vertex on $\Gamma_{12}$.

\begin{lemma}
\label{lem:equivalence-jump}
For all $\bu, \bv \in \bs{H}^1(\Omega_i) \cap \prod_{T\in\mathcal{T}_h(\Omega_i)}\bs{H}^2(T)$ there holds:
\begin{equation}
\label{eq:pironneau-identity-2}
(\nabla\bu^\top - \mathrm{div}\,\bu \cdot \Id, \nabla\bv)_{\Omega_i} = (\nabla\bu^\top \bn - \mathrm{div}\,\bu \cdot \bn,\bv)_{\partial\Omega_i}.
\end{equation}
Moreover, we can equivalently rewrite the surface terms by tangential derivatives as
\begin{equation}
\label{eq:tangential-identity}
(\nabla\bu^\top\bn - \mathrm{div}\,\bu\cdot\bn, \bv)_{\partial\Omega_i}
= -\sum_{j=1}^{n_d}(\nabla\bu\,R_d^j \bn,R_d^j \bv)_{\partial\Omega_i}.
\end{equation}
\end{lemma}
\begin{proof}
The identity \eqref{eq:pironneau-identity-2}, which readily implies \eqref{eq:pironneau-identity}, can be found in \cite[Sect.\,6.4; eq.\,(130)]{pironneau1989finite}. Since the proof is omitted there, we shall give one here for the sake of completeness:
We integrate by parts on each element inside $\Omega_i$ and notice that, due to the regularity requirements, we have $\nabla\mathrm{div}\,\bu - \mathrm{div}\,\nabla\bu^\top = \mathbf{0}$ in every $T \in \mathcal{T}_h (\Omega_i)$. By this, we find after summation that
\begin{align*}
(\nabla\bu^\top - \mathrm{div}\,\bu \cdot \Id, \nabla\bv)_{\Omega_i} = \sum_{T \in
\mathcal{T}_h (\Omega_i)}
 (\nabla\bu^\top \bn - \mathrm{div}\,\bu\cdot\bn, \bv)_{\partial T} .
\end{align*}
Now it is important to see that $\nabla\bu^\top\cdot\bn - \mathrm{div}\,\bu\cdot\bn$ is a derivative acting only in tangential direction, and hence the contributions of each facet in the interior of $\Omega_i$ cancel out due to the inverse orientation of their normals. 
To show this, we make use of the identity \eqref{eq:decomposition} with 
 $T=\nabla\bu^\top$, which directly yields the second identity as a side product.
\end{proof}

Given the above considerations and defining the tangential vectors $\bs{t}_j := R_d^j \bn$, $j=1,\dots,n_d$, and $\bs{v}^\perp_j := R_d^j \bs{v}$, it is easy to see that on $\bs{V} \times \bs{V}$
  the bilinear form $a_{\tr}(\bu,\bv) $ can be equivalently written as
\begin{equation} \label{eq:decouple}
a_{\tr}(\bu,\bv) = (\mu\nabla\bu,\nabla\bv) - \jump{\mu} \sum_{j=1}^{n_d}(\partial_{\bs{t}_j}\bu,\bv^\perp_j)_{\Gamma_{12}} - \mu \sum_{j=1}^{n_d}(\partial_{\bs{t}_j}\bu,\bv^\perp_j)_{\Gamma_{F}}.
\end{equation}

\noindent This formulation allows us to see that the additional couplings at the interface are in fact only introduced in tangential direction, i.e, in conforming velocity spaces, the cross-couplings are restricted to the coefficients at the interface.

\subsection{Equivalence with a modified interface stencil}

Although the representation \eqref{eq:decouple} allows us to 
avoid cross derivative terms within the subdomain, it requires the assembly of interface terms which may not be straightforwardly supported in existing codes. We shall thus shed some light on an alternative interpretation of the interface treatment.

In massively parallel large-scale simulations, it is becoming increasingly important
to develop matrix-free implementations since they do not only
save the memory for storing the full operator but also the effort to retrieve the matrix entries 
in each iteration from memory which is often the dominating 
cost factor. 
In particular, matrix-free implementations are of interest for hierarchically refined grids
that result in patch-wise uniform fine meshes; cf. e.g. \cite{bergen-gradl-ruede-huelsemann_2006,gmeiner-ruede-stengel-waluga-wohlmuth_2015,gmeiner-ruede-stengel-waluga-wohlmuth_2015_2} and the references contained therein. Here, assuming constant coefficients, the interior stencils can be efficiently pre-computed for the coarse mesh elements and can be applied on the fly on the fine mesh. A modified approach is thus in our context only necessary at the interface, where we propose to express the contributions of the bilinear form $a_{\tr} (\cdot,\cdot)$ in a stencil-based fashion.

Let us consider a finite-dimensional space $\mathbf{V}_h\subset\mathbf{V}$. Moreover let us assume that we can decompose $\bv_h = \bv_h^\circ + \widehat{\bv}_h$ such
that $\bv_h^\circ \in \bs{V}_h$ and $\bv_h^\circ|_{\Omega_i} \in
\bs{H}^1_0(\Omega_i),~i=1,2$. It is easy to see that we obtain
\begin{equation}
\label{eq:equivalence-stencil}
a_{\tr}(\bu_h,\bv_h) = (\mu\nabla\bu_h, \nabla\bv^\circ_h)_\Omega + a_{\tr}(\bu_h,\widehat{\bv}_h),
\end{equation}

\noindent where we use \eqref{eq:newbil} for the evaluation of $ a_{\tr}(\bu_h,\widehat{\bv}_h)$.

The conceptual sketch in Fig. \ref{fig:cross} illustrates the coupling of the
x- and y-velocity components in two dimensions. The standard bilinear form
\eqref{eq:standard} depicted in the second picture from the left
results in a fully coupled setting, while the bilinear form \eqref{eq:newbil}
yields a reduced coupling in the components. We note that a closer look reveals that both stiffness matrices resulting from \eqref{eq:decouple}
 and \eqref{eq:equivalence-stencil} are mathematically equal, but differ in their  sparsity structure in a stencil based code which is illustrated by the translucent/green box. 
 
\begin{figure}[ht]
\centering
\includegraphics[width=0.8\textwidth]{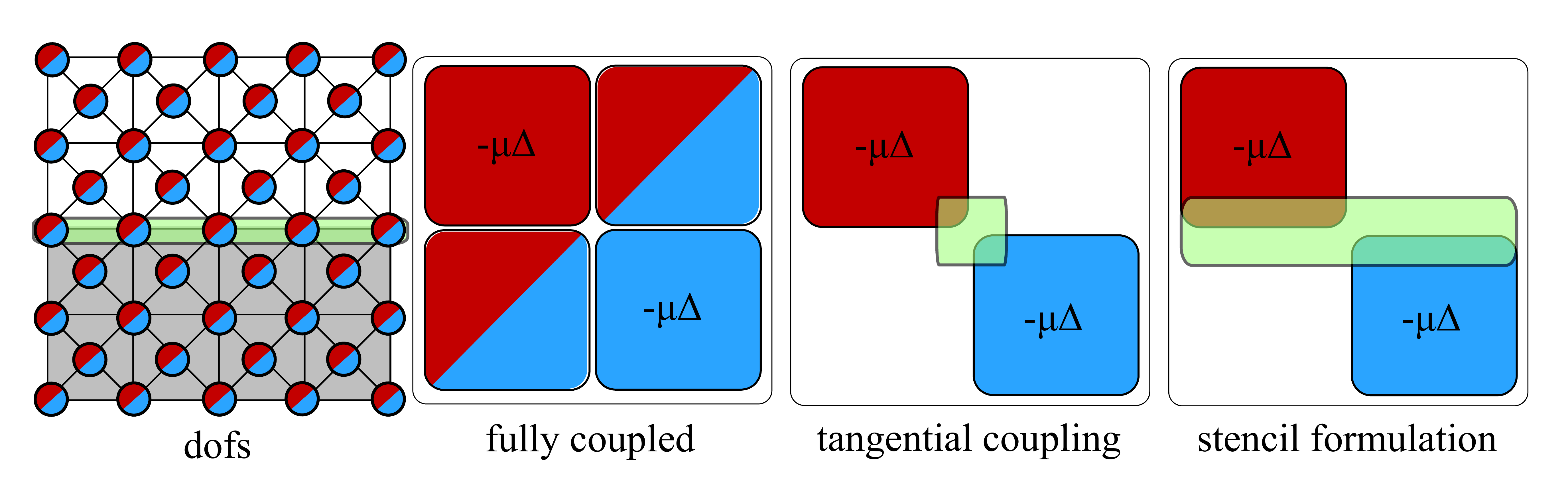}
\caption{An illustration of different types of cross derivative couplings.}
\label{fig:cross}
\end{figure} 

Using \eqref{eq:decouple}, we obtain
a stiffness matrix where only the nodes at the interface couple in x
and y direction, see the second picture from the right. Using the equivalent form \eqref{eq:equivalence-stencil} results in the picture on the right where formally also coupling terms between the subdomain velocity components and interface test functions occur. Although, the element-wise assembled contributions cancel out within the global assembling process, the stencil entry is formally filled by a zero and contributes to the
floating point operations (FLOPS) as well as the memory footprint of the operator application. However, since the additional entries are only relevant at the interface, these extra costs can be asymptotically neglected. From the point of view of stencil-based codes, the application of the viscous operator can thus be realized by applying the Laplacian stencils inside the domains of same viscosity and the full stencil only at the interface $\Gamma_{12} $ between the subdomains with different viscosity and on $\Gamma_F$ 
to ensure consistency. Hence, this decoupling of the velocities inside each subdomain allows for an efficient matrix vector multiplication as required for 
many smoothers in the multigrid context, since each velocity component can be treated independently.

\begin{remark}
The decoupling property is not restricted to settings in which the mesh resolves the interface. For instance, in two-phase flow computations, where the interface is described by a level-set function, the same techniques can be used. In this case, we have to differentiate between shape-functions which are supported at the interface and those which are not. Given such a decomposition, the decoupling mechanisms described above can be straightforwardly applied to increase the sparsity of the global stiffness matrix. The computational effort of 
identifying  the corresponding interface degrees of freedom 
will anyway be needed for 
the element based assembly routines
so that this technique can be directly incorporated without any
significant extra cost. The decoupling property also generalizes to
higher order finite elements.
\end{remark}

%% file: stability.tex
%
By construction, the saddle point problem \eqref{eq:weaks} with $a(\cdot,\cdot)$
given by $a_{\tr}(\cdot,\cdot)$ is consistent.
 We associate with $(\bv,q)  \in \bs{V} \times Q$ the norm
\begin{align*}
\| (\bv, q) \|^2 &:= \| \bv \|_{\bs{V}}^2 +\| q \|_Q^2, \\
 \| \bv \|_{\bs{V}}^2 &:= a_{\sym}(\bv, \bv) , \qquad \| q\|_Q^2 :=(q,q)_Q, \quad (p,q)_Q := ( p/(2 \mu), q)_0 
\end{align*} 

\noindent and consider additionally on $[H^1(\Omega)]^d $ the semi-norm $| \bv |_{\bs{V}}^2 := a_{\dev} (\bv,\bv)$. 
\begin{remark} The choice of our boundary conditions and Korn's inequality 
guarantee that $\| \cdot \|_{\bs{V}}$ is equivalent to the standard $H^1$-Hilbert space norm on $\bs{V}$. However, in general, the constants in the bounds depend on $\mu$.  If on $\Omega_1$, i.e., on the  subdomain with the larger viscosity, a non-trivial Dirichlet boundary part is prescribed,
then the norm $\| \cdot \|_{\bs{V}} $ is equivalent to $(2 \mu \nabla \bv, \nabla \bv)_0^{0.5}$ with constants independent of $\mu$. The inf-sup condition on $\bs{V}\times Q$ is shown in \cite{olshanskii2006analysis} with respect to the norm pairing $(2 \mu \nabla \cdot, \nabla \cdot)_0^{0.5}$ and $\| \cdot \|_Q$, and thus it automatically holds true
for  the pairing $\| \cdot \|_{\bs{V}}$ and $\| \cdot \|_Q$.
\end{remark}

Introducing the subspace of solenoidal velocities $\bs{V}_0 := \{ \bv \in \bs{V}~:~ \div \bv =0 \}$, it is trivial to see that
$$
a_{\sym}(\bv,\bv) = a_{\dev} (\bv,\bv) = a_{\tr}(\bv,\bv), \quad \bv \in \bs{V}_0,$$
and consequently 
 the saddle point problem \eqref{eq:weaks} with $a(\cdot,\cdot)$
given by $a_{\tr}(\cdot,\cdot)$ is stable in the continuous setting. However, the discrete setting does not necessarily inherit this property.

Before we consider stable discretizations of the saddle point problem, we comment on the semi-norm $| \cdot |_{\bs{V}}$. According to \cite[Lemma 4.2]{neff2015poincare}, $| \cdot |_{\bs{V}}$ is equivalent to the standard $H^1$-Hilbert space norm on $\bs{V}$ for $d=3$. Moreover the equivalence  to $(2 \mu \nabla \bv, \nabla \bv)_0^{0.5}$ is uniform with respect to $\mu$ if on $\Omega_1$ a non-trivial Dirichlet boundary part is prescribed. 
However, this does not hold for $d=2$.
In the planar setting, the equivalence of  $| \cdot |_{\bs{V}}$  to  the standard $H^1$-Hilbert
space norm on $\bs{V}$ is only guaranteed if $\Gamma_D = \partial
\Omega$. This difference in the equivalence is related to the fact
that the kernel $\mathcal{N}$ of the operator $\dev \sym \nabla$ is  infinite
dimensional for $d=2$ but finite dimensional for $d=3$. We refer to
\cite{Pom11} for a sequence serving as counterexample in $d=2$. For
$d=3$, the dimension of  $\mathcal{N}$ equals $10$, including the six
dimensional space of rigid body modes.
Since $\div  \mathcal{N} \neq \{ 0 \}$, we can find a $\bv \in  \mathcal{N}$
such that 
$a_{\tr}(\bv,\bv) < 0$. Let us consider as example
$\bv := \mathbf{x}$: we then get $a_{\dev}(\bv,\bv) = 0$ and $a_{\tr}(\bv,\bv) = -3 (\mu_1 | \Omega_1 | +
\mu_2 | \Omega_2|) < 0$. 
In contrast
to  $a_{\dev}(\cdot,\cdot)$, the bilinear form $a_{\tr}(\cdot,\cdot)$ 
is not positive semi-definite on $[H^1(\Omega)]^3$ and does not define a semi-norm.

In the light of the previous considerations, it becomes obvious that the stability of the previously discussed modified schemes is not straightforward to establish if the incompressibility constraint cannot be strongly satisfied given the combination of finite element spaces $\mathbf{V}_h \times Q_h \subset \mathbf{V} \times Q$ at hand. Such mixed spaces, where $\div\mathbf{V}_h=Q_h$ are not very wide-spread in practice since they require higher-order ansatz functions and/or special meshes, cf. e.g. the discussions in \cite{scott1985norm,zhang2005new,Boffi:2012aa,CouBorLeLinRebWan:2013}. In particular, for the case $d=3$, the stability analysis for the so-called Scott--Vogelius spaces has been an open problem for almost three decades; cf. the discussion in \cite{neilan2015discrete} for recent progress.

In the following we shall thus demonstrate techniques that allow to prove stability of our modified scheme for stabilized finite element discretizations.

\subsection{Stabilized lowest-order discretization}
\label{sec:stab-p1p0}

Let us recall that the computational domain $\Omega$ is subdivided into a conforming simplicial mesh $\mathcal{T}_h$ 
that resolves the interface. For the velocity discretization, we use linear, conforming finite elements and for the pressure space we employ piecewise constants, i.e., 
\begin{align*}
\bs{V}_h &:= \{ \bv \in  \bs{V} :  \bv|_T \in [P_1(T)]^d,~\forall \, T \in \mathcal{T}_h \}, 
\\
 Q_h  &:= \{ q \in  Q:  q|_T \in P_0(T),~\forall \, T \in \mathcal{T}_h \} .
\end{align*}
 In terms of a suitable stabilization bilinear form $c(\cdot,\cdot) $,
 the discrete weak problem then reads as: given $\bs{f} \in \bs{V}'$, find $(\bu_h,p_h) \in \bs{V}_h \times Q_h$, such that
\begin{subequations} \label{eq:weakdis}
\begin{align} \label{eq:weak-discrete}
B( \bu_h, p_h; \bv_h,q_h) &= \langle \bs{f},\bv_h \rangle,\qquad \text{for all}~ (\bv_h,q_h) \in \bs{V}_h \times Q_h, \\ \label{eq:defB}
B( \bu_h, p_h; \bv_h,q_h) &:= a_{\rm tr}(\bu_h,\bv_h) + b(\bv_h,p_h) +
b(\bu_h,q_h) - c(p_h,q_h) .
\end{align}
\end{subequations}
The following stability analysis is based on a continuous inf-sup condition for the model problem \cite[Thm.\,1]{olshanskii2006analysis} and the use 
of a subdomain-wise defined Scott-Zhang type operator $S_h$ preserving homogeneous boundary conditions \cite{ScoZha:90}. To obtain a globally conforming approximation, 
we also use a suitable boundary modification
for $S_h$ at the interface $\Gamma_{12}$, i.e., the nodal values of
$S_h \bv $ at each node $p $ on $\bar \Gamma_{12} $ are purely defined
in terms of $\bv |_{\bar \Gamma_{12}}$.
\begin{theorem} \label{thm:stab}
Let the stabilization term $c(\cdot,\cdot) $ satisfy
$$
c(q_h,q_h) \leq \gamma_0 \| q_h \|_Q^2, \quad (q_h, \div (\bs{v} - S_h \bs{v}))_0^2
\leq C_{\text{stab}}^2 c(q_h,q_h) \| \sqrt{2\mu} \nabla \bs{v}\|_0^2
$$
with $ C_{\text{stab}} < \infty$ and $\gamma_0 <  \frac{2d}{d - 2} $ for all $q_h \in Q_h$ and $
\bs{v}|_{\Omega_i} \in [H^1_0(\Omega_i)]^d$. Then the variational formulation \eqref{eq:weakdis}  is uniformly stable with respect to the norm $\| (\bv,q)\|$  and has a unique solution. Moreover the finite element error can be bounded by the best approximation error, i.e.,
\begin{align*}
\| (\bu - \bu_h, p - p_h) \| \leq C \inf_{ (\bv,q) \in \bs{V}_h \times Q_h}
\| (\bu - \bv, p - q) \| .
\end{align*}
\end{theorem}
\begin{proof}
It is sufficient to show that there exists a $(\bv_h,q_h ) \in  \bs{V}_h \times Q_h $ such that
\begin{align} \label{eq:coer}
\frac{B(\bu_h,p_h;\bv_h,q_h)}{\| (\bv_h,p_h)\|} \geq c \| (\bu_h,p_h ) \| .
\end{align}
 We recall that $\mu $ is piecewise constant and moreover  
$( 2 \mu \div\bu_h, 1)_Q = (\div \bu_h,1)_0 =0$ due to the definition of $\bs{V}$.  
Following \cite[Thm. 1]{olshanskii2006analysis}, there exists 
 $ \bs{z} \in \bs{V}$ such that
\begin{subequations}
\begin{align} \label{eq:prop1}
  (\div \bs{z} ,p_h + (1- \frac 2d) \mu \div \bu_h)_0&  = - \| p_h + (1- \frac 2d) \mu \div \bu_h\|_Q^2\\
\label{eq:prop3} 
 \| \sqrt{2 \mu} \nabla \bs{z}\|_0 &\leq C_1 \| p_h   + (1-\frac 2d) \mu \div \bu_h\|_Q
\end{align}
\end{subequations}
with a constant $C_1 < \infty$ independent of $(\bu_h,p_h)$.
Using the assumption on the triangulation, a closer look reveals that $\bs{z}$ can assume the special form 
$\bs{z} = \bs{z}_h + \bs{z}_0 $ with $\bs{z}_h \in \bs{V}_h $, $\bs{z}_0 |_{ \Omega_i} \in [H_0^1
(\Omega_i)]^2$, $i=1,2$ and $ \| \sqrt{2 \mu} \nabla \bs{z}_0\|_0  \leq
 \tilde C_{\text{stab}}\| \sqrt{2 \mu} \nabla \bs{z}\|_0$.
  This falls out of the constructive proof in
\cite{olshanskii2006analysis} and the fact that $\bs{V}_h $ is stable with respect to the one-dimensional pressure subspace of $Q_h$ being subdomain-wise constant. 
For both $d=2$ and $d=3$, we formally set
 $\bv_h := \bu_h + \alpha S_h \bs{z}$ and show \eqref{eq:coer} if $\alpha$ and $q_h$ are  selected properly.  

Now, we first consider the case $d=3$. 
Using  $q_h := - p_h - 2/3 \mu \div\bu_h  \in Q_h$, the definition \eqref{eq:defB} gives
\begin{align*}
B(\bu_h,p_h;\bv_h,q_h) & =  a_{\dev} ( \bu_h , \bu_h ) -
\frac 13 ( \mu \div \bu_h, \div \bu_h)_0  + c(p_h,p_h)
+ \frac 23 c(p_h, \mu \div \bu_h)\\ & +   \alpha a_{\dev}(  \bu_h ,  S_h \bs{z}) - \alpha 
\frac 13 ( \mu \div \bu_h, \div S_h \bs{z})_0- \alpha (\div S_h \bs{z},p_h)_0 \\
& + (\div \bu_h, p_h)_0  + \frac 23  ( \mu \div \bu_h, \div \bu_h)_0 - 
 (\div \bu_h, p_h)_0 \\
& =  a_{\dev}( \bu_h ,  \bu_h ) +
\frac 13 ( \mu \div \bu_h, \div \bu_h)_0 + c(p_h,p_h) + \frac 23 c(p_h, \mu \div \bu_h)\\ & +   \alpha a_{\dev}(  \bu_h ,
 S_h \bs{z}) - \alpha 
 ( p_h +\frac 13 \mu  \div \bu_h, \div S_h \bs{z})_0.
\end{align*}
Exploiting the fact that $S_h$ is subdomain-wise defined
and $H^1$-stable, we find in terms of \eqref{eq:prop3} that
$
| S_h \bs{z} |_{\bs{V}} \leq C_2 \| p_h   + 1/3 \mu \div \bu_h\|_Q$, 
where $C_2 $ only depends on the $H^1$-stability constant of $S_h$ and $
C_1$. Using \eqref{eq:prop1}, we get
\begin{align*}
B(\bu_h,p_h;\bv_h,q_h)
&\geq \frac 12 |  \bu_h |_{\bs{V}}^2 + \frac 13 \| \sqrt{\mu}
\div \bu_h \|_0^2
+ \alpha ( 1 - \frac 12 C_2^2  \alpha ) \| p_h + \frac 13 \mu \div \bu_h \|_Q^2 
\\& + c(p_h,p_h) + \frac 23 c(p_h, \mu \div \bu_h) - \alpha (p_h + \frac 13 \mu \div \bu_h ,
\div (S_h \bs{z} - \bs{z}))_0.
\end{align*}
Due to the linear character of $S_h$, we have $S_h \bs{z}= \bs{z}_h +
S_h \bs{z}_0$ and $ \bs{z} - S_h \bs{z} = \bs{z}_0 -S_h \bs{z}_0$,
and thus we can make use of the assumption on the stabilization term.
 Setting $C_3 := C_1 C_{\text{stab}} \tilde  C_{\text{stab}}^2 /2$, a straightforward computation yields
\begin{align*}
B(\bu_h,p_h;\bv_h,q_h)
&\geq \frac 12 | \bu_h |_{\bs{V}}^2 + \frac 13 \| \sqrt{\mu} \div \bu_h \|_0^2
+ \alpha ( 1 - \frac 12 (C_2^2+ C_3^2)  \alpha   ) \| p_h + \frac 13 \mu  \div \bu_h \|_Q^2 
\\& + c(p_h,p_h) + \frac 23 c(p_h, \mu \div \bu_h) - c (p_h + \frac 13 \mu \div \bu_h ,
p_h + \frac 13 \mu \div \bu_h ) \\
&=
\frac 12 |  \bu_h |_{\bs{V}}^2 + \frac 16 (\| 2\mu \div \bu_h \|_Q^2 - \frac 16
c(2\mu \div \bu_h, 2\mu \div \bu_h)) \\
&+ \alpha ( 1 - \frac 12 (C_2^2+ C_3^2)  \alpha   ) \| p_h + \frac 13 \mu  \div \bu_h \|_Q^2 .
\end{align*}
Under the continuity assumption on the bilinear form $c(\cdot,\cdot) $
  and by selecting $\alpha$ small enough, we find
$$
B(\bu_h,p_h;\bv_h,q_h)  \geq c ( |   \bu_h |_{\bs{V}}^2 +  \| p_h  \|_Q^2 + \| 2 \mu \div \bu_h\|_Q^2 ) \geq c ( \|   \bu_h \|_{\bs{V}}^2 +  \| p_h  \|_Q^2  ) , $$
 whereas $\|(\bv_h, q_h)\| \leq c ( \|\bu_h \|_{\bs{V}} + \| p_h \|_Q +
\| 2 \mu \div \bu_h\|_Q) \leq c ( \|\bu_h \|_{\bs{V}} + \| p_h \|_Q )$.

The proof for $d=2$ with the weaker assumption on $\gamma_0$, i.e., $\gamma_0<\infty$, follows essentially the same lines. Here we instead use $q_h := - p_h - \beta \mu \div\bu_h  \in Q_h$ with $\beta > 0$ small enough and exploit the fact that $a_{\tr} (\cdot,\cdot) = a_{\dev} (\cdot,\cdot)$. \end{proof}

\begin{remark} We point out that the upper bound for $\gamma_0 $ in Theorem
\ref{thm:stab} depends strongly on $d$. For $d=2$, only continuity of $c(\cdot,\cdot)$ with respect to the pressure norm is required.
\end{remark}

\begin{remark}
The proof shows that if on $\Omega_1$ a non-trivial Dirichlet boundary
part is prescribed, the assumption $\gamma_0 < 6$ in 3D can be weakened
since then $| \cdot |_{\bs{V}} $ and $ \| \cdot \|_{\bs{V}} $ are
equivalent and $\| 2 \mu \div \bv \|_Q $ is bounded by $| \bv
|_{\bs{V}} $. Exploiting this equivalence, we
 find an upper bound for $\gamma_0$ being larger
than 6 but depending on a domain-dependent  Korn-type constant which is in general unknown.
\end{remark}

\subsection{An interior-penalty stabilization and local postprocessing}
For $c(\cdot,\cdot) $ there are different options, e.g., we could choose a form which is related to the schemes discussed in \cite{hughes1987new}:
\begin{align}
\label{eq:stabc}
 c(p, q) := \gamma \sum_{i=1}^2  \frac{1}{2 \mu_i} \sum_{F\in\mathcal{F}_h (\Omega_i)}
 \frac{|T_1^F|\cdot|T_2^F|}{|T_1^F|+|T_2^F|}
 \jump{p}|_F \jump{q}|_F, \quad \gamma > 0
\end{align}
where $\mathcal{F}_h (\Omega_i)$
 is the set of interior facets of the isoviscous
subdomain $\Omega_i$ (i.e., facets on  $\partial\Omega_i$ are excluded), and
$T_1^F$, $T_2^F \in {\mathcal T}_h (\Omega_i)$ such that $\partial T_1^F
\cap T_2^F = \bar F$. We point out that we 
do not penalize a pressure jump across
the interface $\Gamma_{12}$.

A straightforward
computation shows that if $\gamma < 3/2$ for $d=3$ then the continuity assumption
of Theorem \ref{thm:stab} is satisfied. Integration by part and the
properties of the Scott--Zhang operator with a proper scaling argument
guarantee the second assumption on $c(\cdot,\cdot)$.

\begin{remark}
Let ${\mathcal T}_h$ be a triangulation which is obtained by uniformly refining 
a mesh ${\mathcal T}_{4h}$ twice according to the procedure described in \cite{Bey:1995fk}. Then, since the pair $\bs{V}_h \times Q_{4h}$ is uniformly inf-sup stable, we can exclude those sub-faces of $\mathcal{F}_h$, which result from a uniform refinement of the faces of the base mesh ${\mathcal T}_{4h}$. Consequently, the choice of $\bs{z}$ in the proof of Theorem \ref{thm:stab} has to be 
adapted, and $\bs{z}_0$ has to be selected such that it is equal to zero on the boundaries of all elements in  ${\mathcal T}_{4h}$. 
\end{remark}

As already pointed out hierarchically refined grids which result in patch-wise uniform fine meshes are of interest for massively parallel large scale simulations. Although stabilized $P_1-P_0$ elements perform well with respect to computational cost and accuracy, the stabilization results in a loss of strong mass conservation. This turns out to be a possible disadvantage in coupled multi-physics problems where the  velocity enters into an advective term. However, using a simple flux correction, we can easily recover exact element-wise mass-conservation in a local post-process.
For convenience of the reader, we shall outline it here for completeness:
For each facet $F$, we define an oriented facet flux by
$$
j_{F;T} := \bu_h \cdot \bn_T  +  \Delta j_{F;T} , \quad \Delta j_{F,T} :=
\frac{\gamma}{2 \mu | F|}
 \left(\frac{|T|\cdot|T^F|}{|T|+|T^F|}\right)
 (p_h |_T - p_h |_{T^F}), 
$$
where $ \bn_T$ is the outer unit normal with respect to the element $T$ and
$\partial T \cap \partial T^F = F$. Then, definition \eqref{eq:stabc} guarantees  local mass conservation, i.e.,
$\sum_{F \subset T} \int_F j_{F;T} \ ds =0$. For many coupled problems it is sufficient to have a weak mass balance, i.e., to have access to $\int_F j_{F;T} \ ds$. If this is not the case, we can also lift the post-processed facet flux onto a BDM element; cf. e.g. \cite{boffi2013mixed} for further reading. In the 
case of continuous pressure approximations, this type of local post-process is also possible \cite{GWW14}.

\subsection{Possible generalizations}

So far we restricted ourselves to the stabilized $P_1-P_0$ setting.
However, as the proof of Theorem \ref{thm:stab} shows, we did not directly exploit the fact that we had been working with a piecewise constant pressure and a conforming piecewise linear velocity. The characteristic feature for the proof 
of our considered pairing is that we have  $\div \bs{V}_h \subset Q_h$.
Thus with minor modifications, the proof can be generalized to higher order stabilized 
$P_k - P_{k-1} $ pairings provided that a suitable stabilization is specified.

Let us also briefly discuss the extension to other mesh types: We assume that ${\mathcal T}_h$ is a family of uniformly refined conforming hexahedral (quadrilateral) meshes and that we use trilinear (bilinear) finite elements in 3D (2D) for the velocity. 
The main difference between the $Q_1-P_0$ and the $P_1-P_0$ setting is that for
the former case we do not have $\div \bs{V}_h \subset Q_h$. As a
consequence, we cannot directly apply the techniques of the proof of
Theorem \ref{thm:stab}. However, it is possible to adapt similar techniques also for the case of piecewise $d$-linear elements. Let us restrict ourselves to affinely mapped hexahedral elements, i.e., each element of ${\mathcal T}_h $ is affinely equivalent to
the reference element $\hat T := (-1,1)^d, d=3$. The following lemma guarantees a local equivalence of
semi-norms.
\begin{lemma} 
Let $\hat \bv_h \in [Q_1(\hat T)]^d$, $d=3$, then it holds
\begin{align} \label{eq:q1p0}
2 \| \dev \sym\nabla \hat \bv_h \|_{0;\hat T}^2
\leq \frac 53 \left( 2 \| \dev \sym\nabla \hat \bv_h \|_{0;\hat T}^2 -
  \frac  13 \| \div \hat \bv_h  - \Pi_0 \div \hat \bv_h  \|_{0;\hat T}^2   \right) .
\end{align}
\end{lemma}
\begin{proof}
 A straightforward
 computation gives the required upper bound by solving a 24-dimensional generalized symmetric eigenvalue problem.
\end{proof} 

Using the affine equivalence, the upper bound \eqref{eq:q1p0}, and
replacing $\div \bu_h $ by  $\Pi_0 \div \bu_h $, in the definition of
$q_h $ and $ \bs{z}$ in the proof of Theorem \ref{thm:stab}, we can
easily adapt it
for the $Q_1-P_0$ case.

%% file: numerics.tex
%
Let us demonstrate the theoretical considerations by some illustrative examples.
We study different test scenarios including traction boundaries and large viscosity jumps as well as several discretizations. 
In large scale computations, the use of stabilized equal-order $P_1$ discretizations (e.g. Brezzi and Pitkäranta, \cite{brezzi1984stabilization}) are quite popular. Although they are formally not covered by our theory, we also study these discretizations below and show by numerical examples that no stability problems occur. Furthermore, we use a higher order discretization, namely a fourth order Scott-Vogelius element, which guarantees strong mass conservation and thus is automatically stable within our new formulation.
Most of the implementations used in the following sections are based on the FEniCS (v\,1.5.0) finite element framework \cite{LoggWellsEtAl2012a}. However, the results for the stabilized equal-order $P_1$ discretization are obtained by a memory-efficient matrix-free implementation in the hierarchical hybrid grids (HHG) framework \cite{BBergen_FHuelsemann_2004,bergen-gradl-ruede-huelsemann_2006,CPE:CPE2968,gmeiner-ruede-stengel-waluga-wohlmuth_2015,gmeiner-huber-john-ruede-wohlmuth_2015}.

\subsection{Traction-type boundary conditions}
In this simple example, we shall demonstrate the effect of the modified bilinear form for applying traction boundary conditions which are not natural to incorporate in a standard gradient-based form of the Stokes problem. For simplicity, we consider $\mu=1$ everywhere inside a rectangular domain $\Omega=[0,5]\times[-1,1]$, and we define a partition of the boundary into an outflow part $\Gamma_{\rm out}={5}\times[-1,1]$, an inflow part $\Gamma_{\rm in}={5}\times(-1,1)$ and a no-slip boundary $\Gamma_{\rm ns}=\partial\Omega\backslash\overline{\Gamma_{\rm in}\cup\Gamma_{\rm out}}$. Given $\bu_{\rm in}(x,y)=(1-y^2,0)^\top$, the strong problem consists of finding $(\mathbf{u},p)$ such that
\begin{subequations}
\label{eq:stokes-strong-traction-bc}
\begin{align}
-\Delta\bu + \nabla p &= \bs{0}, &&\text{in}~\Omega,\\
\div\bu &= 0, &&\text{in}~\Omega,\\
\bu &= \bu_{\rm in}, &&\text{on}~\Gamma_{\rm in},\\
\bu &= \mathbf{0}, &&\text{on}~\Gamma_{\rm ns},\\
\widehat{\bs{\sigma}}(\bu,p)\cdot\bn & = 0, && \text{on}~\Gamma_{\rm out}.
\end{align}
\end{subequations}

\noindent For comparison, we consider $\widehat{\bs{\sigma}}(\bu,p) =
\nabla\bu - p I$ for the natural homogeneous Neumann boundary
conditions for the gradient-based variant of Stokes' problem (we use
$a(\cdot,\cdot)=(\nabla\cdot,\nabla\cdot)$ in the weak form), and we
choose $\widehat{\bs{\sigma}}(\bu,p) = \bs{\sigma}(\bu,p)$ for the
case of traction-free boundary conditions (we use
$a(\cdot,\cdot) = a_{\tr}(\cdot,\cdot)$ in the weak form). As a
reference solution, we also consider the same problem, using the
standard strain-based weak formulation which is compatible with traction-free
boundaries (we use $a(\cdot,\cdot) = a_{\rm sym}(\cdot,\cdot)$
in the weak form). For all cases, we discretize our problem using the fourth order
Scott--Vogelius element ($P_4-P_3^{\rm disc}$). This element is uniformly inf-sup stable for $d=2$ and yields strongly divergence-free velocity solutions \cite{scott1985norm}. The mesh is a
uniform mesh with $50\times20\times2=2\,000$ triangular elements.

Let us firstly remark that the stiffness-matrix for the gradient-based formulation and the stencil-modification only differs at the degrees of freedom associated with the outflow boundary. Hence, in a practical implementation the cross-derivatives only have to be considered in this region for consistency. Secondly, we mention that in contrast to the approach discussed in \cite{LimIdeRosOna:2007}, we obtain a symmetric problem with provable stability and convergence. 
We finally report that for the above example the number of non-zeros in the viscous block only increases by around $0.75 \%$ compared to the saddle-point problem with the gradient-based operator, while the corresponding block in the strain-based Stokes operator requires roughly $89 \%$ more non-zeros. This reduction  in the number of nonzero 
stencil entries leads to both,
lower memory requirements and significant 
savings in terms of floating point operations
that are required for the operator application. 
The savings are even greater 
in a 
three-dimensional 
setting.

\begin{figure}[htbp]
\centerline{\includegraphics[width=.13\textwidth]{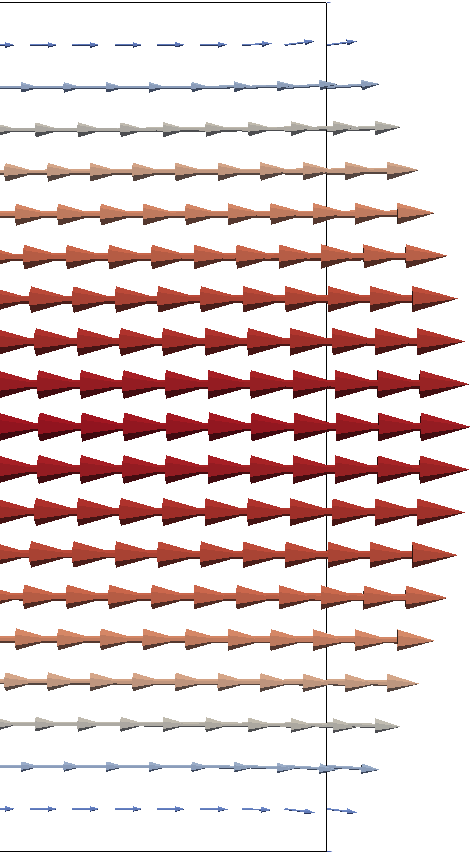}\qquad
\includegraphics[width=.13\textwidth]{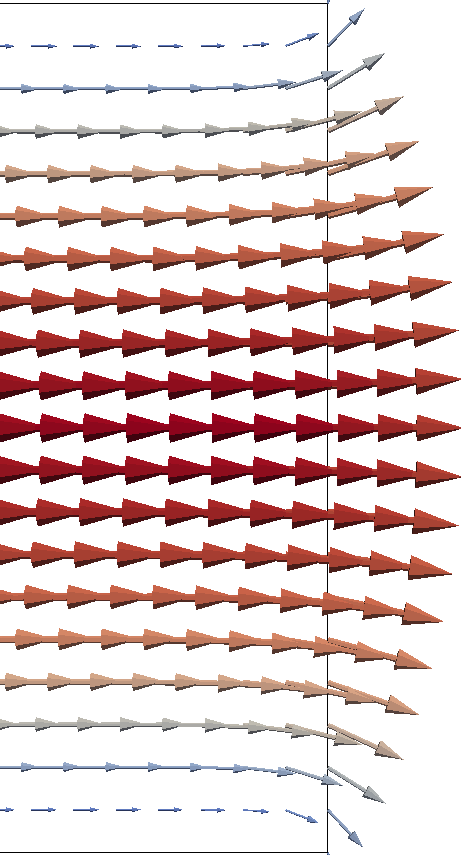}\qquad
\includegraphics[width=.13\textwidth]{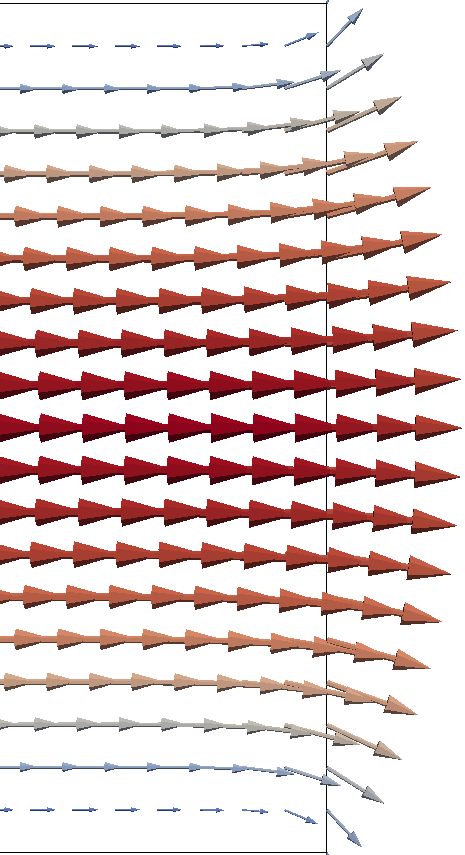}}
\caption{\label{fig:taylor-hood-traction}
Outflow boundary: Plot of the velocity vectors in the outflow region for the gradient-based formulation (left), the stencil-modification (center) and the strain-based formulation (right).}
\end{figure}

In Fig.~\ref{fig:taylor-hood-traction} we depict the outflow part of the domain for all three cases. We clearly observe that the boundary modification allows us to impose traction-free boundary conditions in a consistent manner also for a gradient-based formulation of the Stokes problem. Thus, the choice of the bilinear form does not dictate which boundary conditions shall be used in the whole outflow region. In particular, it is also possible to mix these two kinds of outflow conditions in a single code, giving more flexibility in modeling the flow-conditions that can be assumed past the outflow boundaries.

\subsection{Interface problems}

It is not straightforward to derive exact solutions for the three-dimensional interface problem and, in general,   no closed-form solutions exist. Thus, we consider here a very simple benchmark problem extended from the two-dimensional case and  more complex problems with numerically obtained reference solutions.
 If not mentioned otherwise, we shall below use the stabilized $P_1-P_0$ discretization of Section~\ref{sec:stab-p1p0} together with the stabilization form \eqref{eq:stabc} and $\gamma=1$.

\subsubsection{Couette flow}
For the analytical benchmark, we consider the exact solution of Couette flow given in \cite{HanLarZah:2012:preprint}: The domain is chosen as $\Omega = (0,1) \times (-\tfrac12,\tfrac12) \times (0,1)$ and the interface $\Gamma$ is placed at $y = 0$ dividing $\Omega$ into $\Omega_1=(0,1)\times(-\tfrac12,0) \times (0,1)$ and $\Omega_2=(0,1)\times(0,\tfrac12) \times (0,1)$. We then set $\bs{f} = (3\mu,0,0)^\top$ and boundary conditions such that the exact solution is given by
\begin{align*}
\bu = ( \tfrac12(1-x^2), xy, 0 )^\top, \qquad p(x,y) = 2x\mu - \tfrac12(\mu_1 + \mu_2),
\end{align*}
where $\mu_1=1$ and $\mu_2 = \in \{ 10^{-3}\}$. Our initial mesh is an interface-resolving uniform tetrahedral mesh with mesh-size $h=\tfrac14$, and we study the convergence rates under uniform refinement, see Table~\ref{tab:couette-results-1e3}. As our numerical studies indicate, the expected optimal convergence rates are obtained.

\begin{table}[htbp]
\centering\small
\begin{tabular}{c||cc|cc|cc}
\hline
level & $\|\bu-\bu_h\|_{0}$ & rate & $\|\bu-\bu_h\|_{\mathbf{V}}$ & rate & $\|p-p_h\|_Q$ & rate \\
\hline

0 & 9.5263e-03 &  --  & 1.1065e-01 &  --  & 6.4989e-01 &  --  \\
1 & 2.4913e-03 & 1.93 & 5.3663e-02 & 1.04 & 1.3437e-01 & 2.27 \\
2 & 6.4259e-04 & 1.95 & 2.6567e-02 & 1.01 & 3.9015e-02 & 1.78 \\
3 & 1.6373e-04 & 1.97 & 1.3256e-02 & 1.00 & 1.5737e-02 & 1.30 \\
4 & 4.1456e-05 & 1.98 & 6.6258e-03 & 1.00 & 7.3621e-03 & 1.09 \\
\hline
\end{tabular}
\caption{\label{tab:couette-results-1e3}
Results for the non-isoviscous Couette flow example: $\mu_1/\mu_2=10^3$
}
\end{table}

\subsubsection{Viscosity layers}\label{ssub:visc_layers}

In this test, we consider a  more complex problem for which 
we have no analytical solution available. 
We again use the unit cube $\Omega := (0,1)^3$ and define four subdomains $\Omega_i := (0,1)^2\times (i/4, (i+1)/4)$, $i=0,\dots 3$, where we set the viscosities in each of these subdomains to $\mu_i = 10^i$, i.e., each additional layer increases the viscosity by an order of magnitude. For the boundary condition, we apply free-slip conditions, and we choose the forcing as
$\mathbf{f}=(0,0,-\cos(2\pi x)\cos(2\pi y)\sin(\pi z))^\top$. We choose an initial discretization of the domain into $4^3$ cubes which are then each subdivided into six tetrahedra to form the triangulation at the lowest level.

We compare the discrete solutions $(\bs U_h, P_h)$ obtained for the strain-based Stokes problem with cross-derivatives (i.e., using $a_{\rm sym} (\cdot,\cdot)$ in the weak formulation) to the discrete solutions $(\bu_h, p_h)$ obtained using the modified form without global cross-couplings using $a_{\rm tr} (\cdot,\cdot)$ in the weak formulation. Moreover, to demonstrate the effect of the modified bilinear form, we also compare the strain-based formulation with the gradient-based version that is obtained by using $a(\cdot,\cdot)=(\mu\nabla\cdot,\nabla\cdot)$ in the weak formulation. It should be noted that we are comparing numerical solutions here. However, since the strain-based Stokes discretization is a widely-accepted model for non-isoviscous incompressible flow, we are confident that the reference solution suffices to demonstrate that our method converges optimally in more general situations.

\begin{table}[ht]
\centering\small
\begin{tabular}{c||cc|cc|cc}
\hline
level & $\frac{\|\bs U_h-\bu_h\|_0}{\|\bs U_h\|_0}$ & rate & $\frac{\|\bs U_h-\bu_h\|_{\bs V}}{\|\bs U_h\|_{\bs V}}$ & rate & $\frac{\|P_h-p_h\|_Q}{\|P_h\|_Q}$ & rate \\
\hline
0 & 2.9339e-01 & -- & 3.0574e-01 & -- & 1.7826e-01 & -- \\
1 & 3.0164e-01 & -- & 4.1250e-01 & -- & 2.2806e-01 & -- \\
2 & 3.1359e-01 & -- & 4.8968e-01 & -- & 2.5337e-01 & -- \\
3 & 3.1095e-01 & -- & 5.1625e-01 & -- & 2.6207e-01 & -- \\
4 & 3.0922e-01 & -- & 5.2354e-01 & -- & 2.6458e-01 & -- \\
\hline
\end{tabular}
\caption{\label{tab:layers-laplace}
Layer example using stabilized $P_1-P_0$ elements: relative errors obtained using the gradient-based vs. the strain-based Stokes operator on a series of uniformly refined meshes (negative or small rates are not displayed).}
\end{table}

To obtain a better understanding about the magnitude of the errors involved, we divide the discrete error norms by the respective norm of the reference solution. The results listed in Tables \ref{tab:layers-laplace} and \ref{tab:layers-stencil} show that for this seemingly simple setup, the gradient-based form produces results which are up to $50\%$ off in the energy norm and up to $30\%$ off in the $L^2$-norm of the velocity.

\begin{table}[ht]
\centering\small
\begin{tabular}{c||cc|cc|cc}
\hline
level & $\frac{\|\bs U_h-\bu_h\|_0}{\|\bs U_h\|_0}$ & rate & $\frac{\|\bs U_h-\bu_h\|_{\bs V}}{\|\bs U_h\|_{\bs V}}$ & rate & $\frac{\|P_h-p_h\|_Q}{\|P_h\|_Q}$ & rate \\
\hline
0 & 1.7799e-01 &  --  & 1.6976e-01 &  --  & 1.8562e-01 &  -- \\
1 & 9.9843e-02 & 0.83 & 1.0584e-01 & 0.68 & 1.3775e-01 & 0.43 \\
2 & 3.6872e-02 & 1.43 & 5.2343e-02 & 1.01 & 5.9058e-02 & 1.22 \\
3 & 1.0722e-02 & 1.78 & 2.2601e-02 & 1.21 & 2.1634e-02 & 1.44 \\
4 & 2.8306e-03 & 1.92 & 1.0224e-02 & 1.14 & 8.3074e-03 & 1.38 \\
\hline
\end{tabular}
\caption{\label{tab:layers-stencil}Layer example using stabilized $P_1-P_0$ elements: relative errors obtained using the stencil-modified vs. the strain-based Stokes operator.}
\end{table}

The consistently modified form however leads to an equally convergent scheme as the strain-based reference implementation with cross-derivatives. For comparison, we depict the pressure fields obtained using all approaches in Figure~\ref{fig:layer-pressure}.

\begin{figure}[htbp]\centering
\centerline{\includegraphics[width=.75\textwidth]{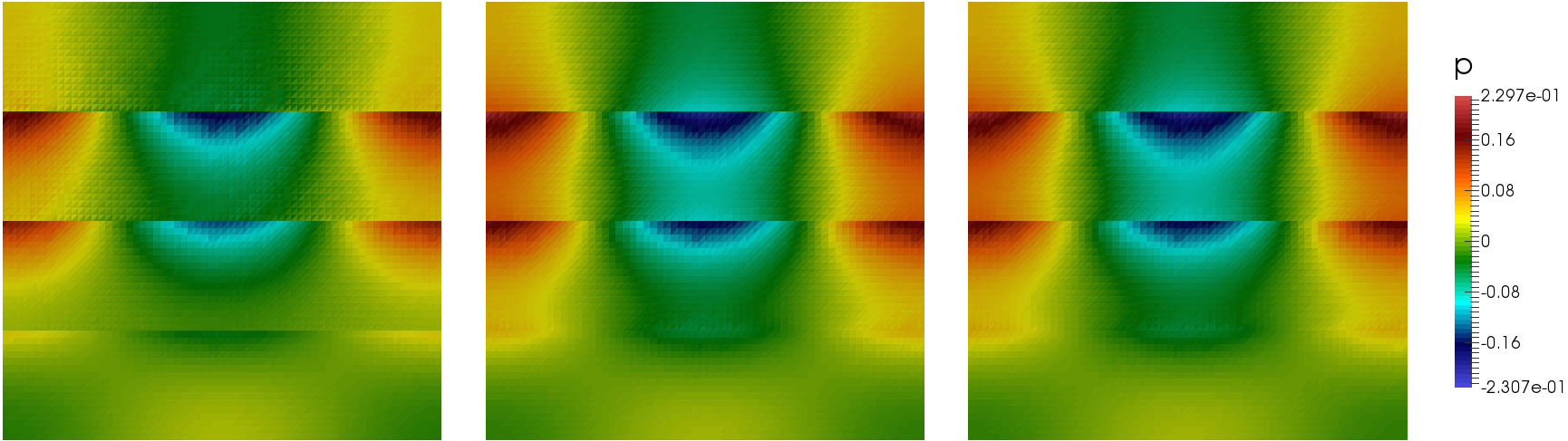}}
\caption{\label{fig:layer-pressure}
Viscosity layers: Plot of the pressure field ($y=0.5$) on the finest mesh level for the gradient-based formulation (left), the stencil modification (center) and the reference solution obtained by the strain-based Stokes formulation (right).}
\end{figure}

\subsubsection{Viscosity columns}\label{ssub:visc_column}

Next, we consider the same domain, mesh-sequence, boundary conditions and forcing terms as in the previous example of Section~\ref{ssub:visc_layers}, but we choose a different viscosity field. We define the columns $\Omega_2 = (0,\tfrac14)^2\times(0,1)\cup(\tfrac34,1)^2\times(0,1)$ and denote the remaining domain by $\Omega_1 := \Omega\backslash\overline{\Omega_2}$. Furthermore, we set $\mu_1=1$ in $\Omega_1$ and $\mu_2=10$ in $\Omega_2$. We note that for this setup, we expect a singularity in the pressure at the edge where the normal is discontinuous, see Figure~\ref{fig:1dplot} for some illustration.

\begin{figure}[ht]
\input{pressure_1dplot}
\input{magu_1dplot}
\vskip-1em
\caption{\label{fig:1dplot} Viscosity columns: Mesh study of the discrete pressure (left), magnitude of the velocity (right) on a 1D line $\gamma(s)$ cutting diagonally through the domain.} 
\end{figure}
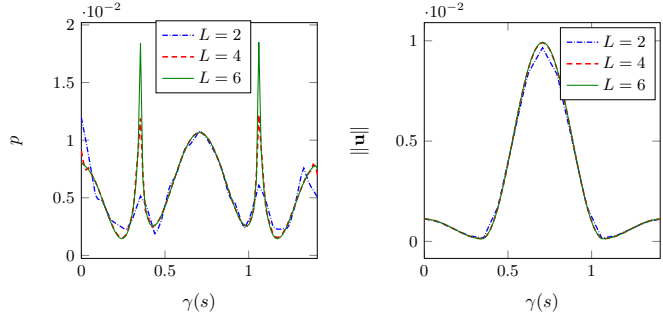

In this example, we first show the errors between the consistent
stencil formulation and the classical strain based formulation on the
same mesh for the theoretically covered $P_1-P_0$ case and for the
popular $P_1-P_1$ discretization which is not covered by our
theory. Note that we use a sparse-matrix implementation in the $P_1-P_0$
case, whereas for the $P_1-P_1$ discretization a memory-efficient
matrix-free implementation is used, for which we compute finer mesh levels, i.e., up to level 6 in Tables~\ref{tab:columns-laplace-p1p1}, \ref{tab:columns-stencil-p1p1} and \ref{tab:columns-stencil-finer-mesh-p1p1}.

\begin{table}[htbp]
\centering\small
\begin{tabular}{c||cc|cc|cc}
\hline
level & $\frac{\|\bs U_h-\bu_h\|_0}{\|\bs U_h\|_0}$ & rate & $\frac{\|\bs U_h-\bu_h\|_{\bs V}}{\|\bs U_h\|_{\bs V}}$ & rate & $\frac{\|P_h-p_h\|_Q}{\|P_h\|_Q}$ & rate \\
\hline
0 & 6.6276e-02 &  --  & 7.1282e-02 &  --  & 9.2322e-02 &  --  \\
1 & 3.2402e-02 & 1.03 & 5.1998e-02 & 0.45 & 7.5894e-02 & 0.28 \\
2 & 2.1200e-02 & 0.61 & 4.3765e-02 & 0.24 & 6.6193e-02 & 0.19 \\
3 & 1.9921e-02 & 0.08 & 4.2069e-02 & 0.05 & 6.5935e-02 & 0.00 \\
4 & 1.9985e-02 & 0.00 & 4.1805e-02 & 0.00 & 6.6048e-02 & 0.00 \\
\hline
\end{tabular}
\caption{\label{tab:columns-laplace}
Column example using stabilized $P_1-P_0$ elements: relative errors obtained using the gradient-based vs. the strain-based Stokes operator.}
\end{table}

 Tables \ref{tab:columns-laplace} and \ref{tab:columns-laplace-p1p1} show the results for the inconsistent gradient based formulation.  We again observe that, independently of the selected discretization, the gradient-based form of the Stokes problem converges to a different solution than the strain-based approach. In this example, the lack of consistency introduces a relative $L^2$ error of roughly $2\%$ which for both discretizations under consideration does not vanish under mesh refinement. 

\begin{table}[ht]
\centering\small
\begin{tabular}{c||cc|cc|cc}
\hline
level & $\frac{\|\bs U_h-\bu_h\|_0}{\|\bs U_h\|_0}$ & rate & $\frac{\|\bs U_h-\bu_h\|_{\bs V}}{\|\bs U_h\|_{\bs V}}$ & rate & $\frac{\|P_h-p_h\|_Q}{\|P_h\|_Q}$ & rate \\
\hline
2 & 7.4550e-02	&	--	&	9.2581e-02	&	--	&	2.7609e01	&	--  \\
3 & 3.0104e-02	&	1.36	&	6.0398e-02	&	0.64	&	1.8003e-01	&	0.64\\
4 & 2.1227e-02	&	0.52	&	4.7364e-02	&	0.36	&	1.1179e-01	&	0.70\\
5 & 2.0227e-02	&	0.07	&	4.3246e-02	&	0.13	&	8.0546e-02	&	0.48\\
6 & 2.0098e-02	&	0.01	&	4.2121e-02	&	0.04	&	6.9997e-02	&	0.20\\
\hline
\end{tabular}
\caption{\label{tab:columns-laplace-p1p1}
Column example using stabilized $P_1-P_1$ elements: relative errors obtained using the gradient-based vs. the strain-based Stokes operator.}
\end{table}

 The results for the consistently modified scheme are listed in Tables~\ref{tab:columns-stencil} and \ref{tab:columns-stencil-p1p1}. As is theoretically expected, the  stencil based formulation eliminates the consistency error in both cases and achieves good rates of convergence in the respective norms.

\begin{table}[htbp]
\centering\small
\begin{tabular}{c||cc|cc|cc}
\hline
level & $\frac{\|\bs U_h-\bu_h\|_0}{\|\bs U_h\|_0}$ & rate & $\frac{\|\bs U_h-\bu_h\|_{\bs V}}{\|\bs U_h\|_{\bs V}}$ & rate & $\frac{\|P_h-p_h\|_Q}{\|P_h\|_Q}$ & rate \\
\hline
0 & 1.1023e-01 &  --  & 1.1881e-01 &  --  & 1.4978e-01 &  --  \\
1 & 5.0381e-02 & 1.12 & 7.1554e-02 & 0.73 & 1.0439e-01 & 0.52 \\
2 & 1.7661e-02 & 1.51 & 3.4368e-02 & 1.05 & 5.0572e-02 & 1.04 \\
3 & 4.9194e-03 & 1.84 & 1.5852e-02 & 1.11 & 2.5225e-02 & 1.00 \\
4 & 1.2678e-03 & 1.95 & 7.5513e-03 & 1.06 & 1.2397e-02 & 1.02 \\
\hline
\end{tabular}
\caption{\label{tab:columns-stencil}
Column example using stabilized $P_1-P_0$ elements: relative errors obtained using the stencil-modified vs. the strain-based Stokes operator.}
\end{table}

\begin{table}[ht]
\centering\small
\begin{tabular}{c||cc|cc|cc}
\hline
level & $\frac{\|\bs U_h-\bu_h\|_0}{\|\bs U_h\|_0}$ & rate & $\frac{\|\bs U_h-\bu_h\|_{\bs V}}{\|\bs U_h\|_{\bs V}}$ & rate & $\frac{\|P_h-p_h\|_Q}{\|P_h\|_Q}$ & rate \\
\hline
2 & 1.3890e-02	&	--	&	2.4189e-02	&	--	&	7.7949e-02	&	--\\
3 & 4.2414e-03	&	1.79	&	1.2312e-02	&	1.02	&	4.8195e-02	&	0.72\\
4 & 1.1370e-03	&	1.94	&	6.2303e-03	&	1.00	&	2.5287e-02	&	0.95\\
5 & 2.9449e-04	&	1.97	&	3.1918e-03	&	0.98	&	1.3274e-02	&	0.94\\
6 & 7.5985e-05	&	1.96	&	1.6675e-03	&	0.94	&	6.8175e-03	&	0.97\\
\hline
\end{tabular}
\caption{ \label{tab:columns-stencil-p1p1}
Column example using stabilized $P_1-P_1$ elements: relative errors obtained using the stencil-modified vs. the strain-based Stokes operator.}
\end{table}

Here, we not only compare the errors on the same mesh levels, but also with respect to a solution obtained on two subsequent refinements of the maximum considered mesh level. Table~\ref{tab:columns-stencil-finer-mesh-p1p1} shows the results for the stabilized $P_1 - P_1$ discretization. We observe a slightly reduced rate for the $L^2$ error of the velocity and the pressure-error, which may be due to the impact of the singularity on the asymptotic order of convergence. It shall also be noted that the difference between the two consistent formulations on the same mesh is roughly by a factor of five smaller than the discretization error that can be estimated with help of the high fidelity solution.

\begin{table}[ht]
\centering\small
\begin{tabular}{c||cc|cc|cc}
\hline
level & $\frac{\|\bs U_h-\bu_h\|_0}{\|\bs U_h\|_0}$ & rate & $\frac{\|\bs U_h-\bu_h\|_{\bs V}}{\|\bs U_h\|_{\bs V}}$ & rate & $\frac{\|P_h-p_h\|_Q}{\|P_h\|_Q}$ & rate \\
\hline
2 & 1.1073e-01	&	--	&	4.8731e-01	&	--	&	6.9077e-01	&	--	\\
3 & 2.0170e-02	&	2.57	&	1.3751e-01	&	1.90	&	1.5605e-01	&	2.24\\
4 & 5.7436e-03	&	1.85	&	6.9474e-02	&	1.01	&	8.2428e-02	&	0.94\\
5 & 1.5953e-03	&	1.86	&	3.4779e-02	&	1.01	&	4.3231e-02	&	0.94\\
6 & 4.2928e-04	&	1.90	&	1.7104e-02	&	1.03	&	2.2758e-02	&	0.93\\
\hline
\end{tabular}
\caption{\label{tab:columns-stencil-finer-mesh-p1p1}
Column example using stabilized $P_1-P_1$ elements: relative errors obtained using the stencil-modified vs. the strain-based Stokes operator. For the reference we take a level 8 solution of the strain-based formulation.
}
\end{table}

%% file: pressure_1dplot.tex
\centering
\pgfplotsset{every tick label/.append style={font=\small}}
\begin{tikzpicture}[scale=0.7]
		\begin{axis}[
		width=0.5\textwidth,
		height=0.5\textwidth,
		xlabel={$\gamma(s)$},
		ylabel={$p$},
		legend style = {at = {(0.315,0.855)},anchor=west,font=\small},
		xmin = 0, xmax=1.41
				]

 \addplot[densely dashdotted,line width=0.25mm,color=blue] coordinates {
(	0	,	0.011977	)
(	0.014142	,	0.010909	)
(	0.028284	,	0.0098412	)
(	0.042426	,	0.0087229	)
(	0.056569	,	0.0075539	)
(	0.070711	,	0.0063848	)
(	0.084853	,	0.0052158	)
(	0.098995	,	0.0048847	)
(	0.11314	,	0.0048328	)
(	0.12728	,	0.0046779	)
(	0.14142	,	0.0042138	)
(	0.15556	,	0.0037497	)
(	0.16971	,	0.0032856	)
(	0.18385	,	0.0029874	)
(	0.19799	,	0.0028551	)
(	0.21213	,	0.0027228	)
(	0.22627	,	0.0025916	)
(	0.24042	,	0.0024604	)
(	0.25456	,	0.0023292	)
(	0.2687	,	0.0023025	)
(	0.28284	,	0.0025894	)
(	0.29698	,	0.0028762	)
(	0.31113	,	0.0033882	)
(	0.32527	,	0.0039752	)
(	0.33941	,	0.0045623	)
(	0.35355	,	0.0051493	)
(	0.3677	,	0.0048098	)
(	0.38184	,	0.0044702	)
(	0.39598	,	0.0039554	)
(	0.41012	,	0.0032654	)
(	0.42426	,	0.0025754	)
(	0.43841	,	0.0018853	)
(	0.45255	,	0.0020901	)
(	0.46669	,	0.0025932	)
(	0.48083	,	0.0031827	)
(	0.49497	,	0.004032	)
(	0.50912	,	0.0048812	)
(	0.52326	,	0.0057305	)
(	0.5374	,	0.0063031	)
(	0.55154	,	0.0065992	)
(	0.56569	,	0.0068953	)
(	0.57983	,	0.0075389	)
(	0.59397	,	0.0081824	)
(	0.60811	,	0.008826	)
(	0.62225	,	0.0093444	)
(	0.6364	,	0.0094876	)
(	0.65054	,	0.0096307	)
(	0.66468	,	0.0098975	)
(	0.67882	,	0.010206	)
(	0.69296	,	0.010514	)
(	0.70711	,	0.010822	)
(	0.72125	,	0.010625	)
(	0.73539	,	0.010429	)
(	0.74953	,	0.010153	)
(	0.76368	,	0.0097969	)
(	0.77782	,	0.0094413	)
(	0.79196	,	0.0090856	)
(	0.8061	,	0.0086817	)
(	0.82024	,	0.0082617	)
(	0.83439	,	0.0077569	)
(	0.84853	,	0.006998	)
(	0.86267	,	0.0062391	)
(	0.87681	,	0.0054802	)
(	0.89095	,	0.0049838	)
(	0.9051	,	0.00475	)
(	0.91924	,	0.0045162	)
(	0.93338	,	0.0039211	)
(	0.94752	,	0.0033261	)
(	0.96167	,	0.002731	)
(	0.97581	,	0.0023828	)
(	0.98995	,	0.0027751	)
(	1.0041	,	0.0031675	)
(	1.0182	,	0.0038281	)
(	1.0324	,	0.0045783	)
(	1.0465	,	0.0053284	)
(	1.0607	,	0.0060785	)
(	1.0748	,	0.0056784	)
(	1.0889	,	0.0052782	)
(	1.1031	,	0.0047105	)
(	1.1172	,	0.0039752	)
(	1.1314	,	0.0032398	)
(	1.1455	,	0.0025045	)
(	1.1597	,	0.0023037	)
(	1.1738	,	0.0022811	)
(	1.1879	,	0.0022719	)
(	1.2021	,	0.002303	)
(	1.2162	,	0.0023341	)
(	1.2304	,	0.0023652	)
(	1.2445	,	0.0027195	)
(	1.2587	,	0.0033969	)
(	1.2728	,	0.0040743	)
(	1.2869	,	0.0050476	)
(	1.3011	,	0.0060208	)
(	1.3152	,	0.0069941	)
(	1.3294	,	0.0075757	)
(	1.3435	,	0.0069826	)
(	1.3576	,	0.0063894	)
(	1.3718	,	0.005987	)
(	1.3859	,	0.005648	)
(	1.4001	,	0.0053091	)
(	1.4142	,	0.0049702	)
};

\addplot[color=red,line width=0.3mm, densely dashed] coordinates {
(	0	,	0.0090734	)
(	0.014142	,	0.0079706	)
(	0.028284	,	0.0074255	)
(	0.042426	,	0.0075884	)
(	0.056569	,	0.007208	)
(	0.070711	,	0.006754	)
(	0.084853	,	0.0063766	)
(	0.098995	,	0.005895	)
(	0.11314	,	0.0053416	)
(	0.12728	,	0.0048037	)
(	0.14142	,	0.0042429	)
(	0.15556	,	0.0036553	)
(	0.16971	,	0.0031458	)
(	0.18385	,	0.0026442	)
(	0.19799	,	0.0021829	)
(	0.21213	,	0.0018749	)
(	0.22627	,	0.001621	)
(	0.24042	,	0.0015125	)
(	0.25456	,	0.0015852	)
(	0.2687	,	0.0018217	)
(	0.28284	,	0.0024183	)
(	0.29698	,	0.0031192	)
(	0.31113	,	0.0041112	)
(	0.32527	,	0.006582	)
(	0.33941	,	0.0091703	)
(	0.35355	,	0.011876	)
(	0.3677	,	0.0071737	)
(	0.38184	,	0.0039302	)
(	0.39598	,	0.0034091	)
(	0.41012	,	0.0028538	)
(	0.42426	,	0.0024685	)
(	0.43841	,	0.0026151	)
(	0.45255	,	0.0028272	)
(	0.46669	,	0.0031186	)
(	0.48083	,	0.0035613	)
(	0.49497	,	0.0040604	)
(	0.50912	,	0.0046178	)
(	0.52326	,	0.0051951	)
(	0.5374	,	0.0058157	)
(	0.55154	,	0.0064531	)
(	0.56569	,	0.0070355	)
(	0.57983	,	0.007655	)
(	0.59397	,	0.0082309	)
(	0.60811	,	0.0087345	)
(	0.62225	,	0.0092373	)
(	0.6364	,	0.0096494	)
(	0.65054	,	0.0099925	)
(	0.66468	,	0.0103	)
(	0.67882	,	0.010485	)
(	0.69296	,	0.010609	)
(	0.70711	,	0.010674	)
(	0.72125	,	0.010609	)
(	0.73539	,	0.010487	)
(	0.74953	,	0.010288	)
(	0.76368	,	0.010002	)
(	0.77782	,	0.0096454	)
(	0.79196	,	0.0092256	)
(	0.8061	,	0.008752	)
(	0.82024	,	0.0082142	)
(	0.83439	,	0.007653	)
(	0.84853	,	0.0070583	)
(	0.86267	,	0.0064242	)
(	0.87681	,	0.0058283	)
(	0.89095	,	0.0052089	)
(	0.9051	,	0.004597	)
(	0.91924	,	0.0040848	)
(	0.93338	,	0.0035675	)
(	0.94752	,	0.0031316	)
(	0.96167	,	0.0028404	)
(	0.97581	,	0.0026144	)
(	0.98995	,	0.0025482	)
(	1.0041	,	0.0028257	)
(	1.0182	,	0.0033552	)
(	1.0324	,	0.0041341	)
(	1.0465	,	0.0071079	)
(	1.0607	,	0.012194	)
(	1.0748	,	0.0093546	)
(	1.0889	,	0.0065224	)
(	1.1031	,	0.0040139	)
(	1.1172	,	0.003171	)
(	1.1314	,	0.0025052	)
(	1.1455	,	0.0018195	)
(	1.1597	,	0.0015983	)
(	1.1738	,	0.0015663	)
(	1.1879	,	0.0016457	)
(	1.2021	,	0.0018759	)
(	1.2162	,	0.0022269	)
(	1.2304	,	0.0026507	)
(	1.2445	,	0.0031437	)
(	1.2587	,	0.0036922	)
(	1.2728	,	0.0042223	)
(	1.2869	,	0.004803	)
(	1.3011	,	0.0053688	)
(	1.3152	,	0.0058566	)
(	1.3294	,	0.0063622	)
(	1.3435	,	0.0068347	)
(	1.3576	,	0.0071198	)
(	1.3718	,	0.0073766	)
(	1.3859	,	0.0079095	)
(	1.4001	,	0.0076874	)
(	1.4142	,	0.0068882	)
};

\addplot[color=green!50!black,line width=0.2mm,solid] coordinates {
(	0	,	0.0080948	)
(	0.014142	,	0.0078202	)
(	0.028284	,	0.0076788	)
(	0.042426	,	0.0074539	)
(	0.056569	,	0.0071509	)
(	0.070711	,	0.0068059	)
(	0.084853	,	0.0063605	)
(	0.098995	,	0.0058654	)
(	0.11314	,	0.0053386	)
(	0.12728	,	0.0047786	)
(	0.14142	,	0.0042078	)
(	0.15556	,	0.0036387	)
(	0.16971	,	0.0030956	)
(	0.18385	,	0.0025937	)
(	0.19799	,	0.0021533	)
(	0.21213	,	0.001804	)
(	0.22627	,	0.0015485	)
(	0.24042	,	0.00145	)
(	0.25456	,	0.0015148	)
(	0.2687	,	0.001754	)
(	0.28284	,	0.0022465	)
(	0.29698	,	0.0030322	)
(	0.31113	,	0.0042205	)
(	0.32527	,	0.0060364	)
(	0.33941	,	0.0091375	)
(	0.35355	,	0.018411	)
(	0.3677	,	0.0066945	)
(	0.38184	,	0.0042305	)
(	0.39598	,	0.0032094	)
(	0.41012	,	0.002714	)
(	0.42426	,	0.0025346	)
(	0.43841	,	0.0025846	)
(	0.45255	,	0.0027916	)
(	0.46669	,	0.0031242	)
(	0.48083	,	0.0035556	)
(	0.49497	,	0.0040519	)
(	0.50912	,	0.0046059	)
(	0.52326	,	0.0051981	)
(	0.5374	,	0.0058115	)
(	0.55154	,	0.0064371	)
(	0.56569	,	0.0070525	)
(	0.57983	,	0.0076517	)
(	0.59397	,	0.0082215	)
(	0.60811	,	0.0087545	)
(	0.62225	,	0.0092298	)
(	0.6364	,	0.0096519	)
(	0.65054	,	0.010008	)
(	0.66468	,	0.010289	)
(	0.67882	,	0.010498	)
(	0.69296	,	0.01062	)
(	0.70711	,	0.010661	)
(	0.72125	,	0.010622	)
(	0.73539	,	0.010504	)
(	0.74953	,	0.010286	)
(	0.76368	,	0.010002	)
(	0.77782	,	0.009656	)
(	0.79196	,	0.009234	)
(	0.8061	,	0.00875	)
(	0.82024	,	0.0082233	)
(	0.83439	,	0.0076537	)
(	0.84853	,	0.0070538	)
(	0.86267	,	0.0064328	)
(	0.87681	,	0.0058152	)
(	0.89095	,	0.0051983	)
(	0.9051	,	0.0046086	)
(	0.91924	,	0.0040535	)
(	0.93338	,	0.0035501	)
(	0.94752	,	0.0031265	)
(	0.96167	,	0.0027963	)
(	0.97581	,	0.0025869	)
(	0.98995	,	0.0025388	)
(	1.0041	,	0.0027116	)
(	1.0182	,	0.0032077	)
(	1.0324	,	0.0042468	)
(	1.0465	,	0.006684	)
(	1.0607	,	0.018495	)
(	1.0748	,	0.0091489	)
(	1.0889	,	0.0060443	)
(	1.1031	,	0.0042246	)
(	1.1172	,	0.0030371	)
(	1.1314	,	0.0022438	)
(	1.1455	,	0.0017538	)
(	1.1597	,	0.001499	)
(	1.1738	,	0.0014445	)
(	1.1879	,	0.001551	)
(	1.2021	,	0.0017971	)
(	1.2162	,	0.0021428	)
(	1.2304	,	0.0025878	)
(	1.2445	,	0.0030864	)
(	1.2587	,	0.0036289	)
(	1.2728	,	0.0041947	)
(	1.2869	,	0.0047534	)
(	1.3011	,	0.0053212	)
(	1.3152	,	0.0058627	)
(	1.3294	,	0.0063359	)
(	1.3435	,	0.0067612	)
(	1.3576	,	0.0071566	)
(	1.3718	,	0.0074495	)
(	1.3859	,	0.0076928	)
(	1.4001	,	0.0077955	)
(	1.4142	,	0.0074643	)
};
\legend{$L=2$, $L=4$, $L=6$}
	\end{axis}
	\end{tikzpicture}

%% file: magu_1dplot.tex
\centering
\pgfplotsset{every tick label/.append style={font=\small}}
	\begin{tikzpicture}[scale=0.7]
		\begin{axis}[
		width=0.5\textwidth,
		height=0.5\textwidth,
		xlabel={$\gamma(s)$},
		ylabel={$\| \bf u\|$},
		legend style = {font=\small},
		xmin = 0, xmax=1.41
				]

 \addplot[densely dashdotted,line width=0.25mm,color=blue] coordinates {

(	0	,	0.00113340000011539	)
(	0.014142	,	0.0011055085797391	)
(	0.028284	,	0.0010777352041158	)
(	0.042426	,	0.00105040617778696	)
(	0.056569	,	0.00102355906629075	)
(	0.070711	,	0.000996715978135196	)
(	0.084853	,	0.000969982987370913	)
(	0.098995	,	0.000921190324341827	)
(	0.11314	,	0.000865051821475454	)
(	0.12728	,	0.000809601082452957	)
(	0.14142	,	0.00075618662654202	)
(	0.15556	,	0.000702821254257439	)
(	0.16971	,	0.000649527137864154	)
(	0.18385	,	0.00059525377422575	)
(	0.19799	,	0.000539988182994776	)
(	0.21213	,	0.000484734863928725	)
(	0.22627	,	0.00043214466852664	)
(	0.24042	,	0.000379560474039382	)
(	0.25456	,	0.00032698528249143	)
(	0.2687	,	0.000282372854959537	)
(	0.28284	,	0.000261644770824872	)
(	0.29698	,	0.000240929212950609	)
(	0.31113	,	0.000220973704130152	)
(	0.32527	,	0.000201516002751767	)
(	0.33941	,	0.00018242311507701	)
(	0.35355	,	0.000163812615054153	)
(	0.3677	,	0.000421421379393595	)
(	0.38184	,	0.000679363403660368	)
(	0.39598	,	0.000934266584921028	)
(	0.41012	,	0.00118724870119954	)
(	0.42426	,	0.00144107648166223	)
(	0.43841	,	0.00169527064376164	)
(	0.45255	,	0.00216677323698166	)
(	0.46669	,	0.00271407508142277	)
(	0.48083	,	0.00325614985381509	)
(	0.49497	,	0.00377769448845457	)
(	0.50912	,	0.00429970968357167	)
(	0.52326	,	0.00482187892636263	)
(	0.5374	,	0.00535183759369621	)
(	0.55154	,	0.0058901899148075	)
(	0.56569	,	0.00642942618528279	)
(	0.57983	,	0.00693592115942072	)
(	0.59397	,	0.00744304750639817	)
(	0.60811	,	0.00795086531329138	)
(	0.62225	,	0.00838708385155413	)
(	0.6364	,	0.00860749316252996	)
(	0.65054	,	0.00882807690272349	)
(	0.66468	,	0.00903717863402622	)
(	0.67882	,	0.00924297639899616	)
(	0.69296	,	0.00944950160865455	)
(	0.70711	,	0.00965670784018011	)
(	0.72125	,	0.00943090116096675	)
(	0.73539	,	0.00920527780103246	)
(	0.74953	,	0.00898574223182086	)
(	0.76368	,	0.0087727238428267	)
(	0.77782	,	0.00856062041139543	)
(	0.79196	,	0.0083492941105521	)
(	0.8061	,	0.00788960418377247	)
(	0.82024	,	0.00734724393391835	)
(	0.83439	,	0.00681316809604754	)
(	0.84853	,	0.00630363376744715	)
(	0.86267	,	0.00579488909194991	)
(	0.87681	,	0.00528739012232878	)
(	0.89095	,	0.00475947381613766	)
(	0.9051	,	0.00421041863316464	)
(	0.91924	,	0.00366229221360066	)
(	0.93338	,	0.00314495023911349	)
(	0.94752	,	0.00262851622450766	)
(	0.96167	,	0.00211392184675782	)
(	0.97581	,	0.00166660797705999	)
(	0.98995	,	0.00141295355808321	)
(	1.0041	,	0.00115994366087323	)
(	1.0182	,	0.000909750258147806	)
(	1.0324	,	0.000660463933458898	)
(	1.0465	,	0.000411443290363326	)
(	1.0607	,	0.000163891231614141	)
(	1.0748	,	0.000184091471005585	)
(	1.0889	,	0.000204524407619262	)
(	1.1031	,	0.000224639544465016	)
(	1.1172	,	0.000244605440070269	)
(	1.1314	,	0.000264893772288063	)
(	1.1455	,	0.000285453655594039	)
(	1.1597	,	0.000332120950082948	)
(	1.1738	,	0.00038745667152083	)
(	1.1879	,	0.000442102549904431	)
(	1.2021	,	0.000494675475458406	)
(	1.2162	,	0.000547239141189663	)
(	1.2304	,	0.000599803193860286	)
(	1.2445	,	0.000653882205435352	)
(	1.2587	,	0.000709498787218132	)
(	1.2728	,	0.000765151366729486	)
(	1.2869	,	0.000817734097227821	)
(	1.3011	,	0.000870353013364692	)
(	1.3152	,	0.000923001890377804	)
(	1.3294	,	0.000969209526680893	)
(	1.3435	,	0.000996030716958568	)
(	1.3576	,	0.00102292313398271	)
(	1.3718	,	0.00104844782993194	)
(	1.3859	,	0.00107370754472575	)
(	1.4001	,	0.00109892626925035	)
(	1.4142	,	0.00112430000019967	)
};

\addplot[color=red,line width=0.3mm, densely dashed] coordinates {

(	0	,	0.001112	)
(	0.014142	,	0.00110472639305856	)
(	0.028284	,	0.00109281181903931	)
(	0.042426	,	0.00107495096890602	)
(	0.056569	,	0.00104742723575817	)
(	0.070711	,	0.00101576460017467	)
(	0.084853	,	0.000977159170324876	)
(	0.098995	,	0.00093213100985752	)
(	0.11314	,	0.000884064831163416	)
(	0.12728	,	0.000830209182514865	)
(	0.14142	,	0.000773482042165169	)
(	0.15556	,	0.000715089803536591	)
(	0.16971	,	0.000653761548470082	)
(	0.18385	,	0.000592683113255304	)
(	0.19799	,	0.000531504508088125	)
(	0.21213	,	0.000471765863561576	)
(	0.22627	,	0.000414285331731646	)
(	0.24042	,	0.000359321272362214	)
(	0.25456	,	0.000309469459562652	)
(	0.2687	,	0.000263248625964125	)
(	0.28284	,	0.000223237967536887	)
(	0.29698	,	0.000190460737050973	)
(	0.31113	,	0.000162631611987891	)
(	0.32527	,	0.000145494985252139	)
(	0.33941	,	0.000136872732183514	)
(	0.35355	,	0.000135053073674685	)
(	0.3677	,	0.000241451098465093	)
(	0.38184	,	0.000400384231957254	)
(	0.39598	,	0.00061291897686399	)
(	0.41012	,	0.000906957016566937	)
(	0.42426	,	0.00124148701555836	)
(	0.43841	,	0.00163798588034207	)
(	0.45255	,	0.0020934169420352	)
(	0.46669	,	0.00257823520304103	)
(	0.48083	,	0.00311465978610827	)
(	0.49497	,	0.00367676575835339	)
(	0.50912	,	0.00425396551657392	)
(	0.52326	,	0.00485440774628584	)
(	0.5374	,	0.00545075685962968	)
(	0.55154	,	0.00604573851478874	)
(	0.56569	,	0.00662251504375792	)
(	0.57983	,	0.00717548266273426	)
(	0.59397	,	0.00770182030439558	)
(	0.60811	,	0.00817669897447619	)
(	0.62225	,	0.00861517142993684	)
(	0.6364	,	0.00899371974840221	)
(	0.65054	,	0.00930493031291476	)
(	0.66468	,	0.00957106489518277	)
(	0.67882	,	0.00974429619486189	)
(	0.69296	,	0.00985202213530045	)
(	0.70711	,	0.00991010054242057	)
(	0.72125	,	0.00985102702855707	)
(	0.73539	,	0.00974222815647427	)
(	0.74953	,	0.00956949087125329	)
(	0.76368	,	0.00930025478226269	)
(	0.77782	,	0.00898978655409571	)
(	0.79196	,	0.0086113050979628	)
(	0.8061	,	0.00816925306954069	)
(	0.82024	,	0.00769743161085566	)
(	0.83439	,	0.00716919032194013	)
(	0.84853	,	0.00661348211173206	)
(	0.86267	,	0.00604196585203856	)
(	0.87681	,	0.0054426335842862	)
(	0.89095	,	0.00484642631742194	)
(	0.9051	,	0.00425023811382374	)
(	0.91924	,	0.00366829616835936	)
(	0.93338	,	0.0031089613025575	)
(	0.94752	,	0.00257413336266791	)
(	0.96167	,	0.00208736313520192	)
(	0.97581	,	0.00163502148646432	)
(	0.98995	,	0.00123874276183556	)
(	1.0041	,	0.000905040095686373	)
(	1.0182	,	0.000612245537835924	)
(	1.0324	,	0.000401167916090008	)
(	1.0465	,	0.000243765074319518	)
(	1.0607	,	0.000135148144641057	)
(	1.0748	,	0.000137047407336367	)
(	1.0889	,	0.00014574720235092	)
(	1.1031	,	0.000162830269272086	)
(	1.1172	,	0.000190876918253622	)
(	1.1314	,	0.000223635573382233	)
(	1.1455	,	0.00026365683317904	)
(	1.1597	,	0.000310190232489677	)
(	1.1738	,	0.000359803642363442	)
(	1.1879	,	0.000414948798193223	)
(	1.2021	,	0.000472703292190778	)
(	1.2162	,	0.000531966557031549	)
(	1.2304	,	0.000593561703053693	)
(	1.2445	,	0.00065465088235028	)
(	1.2587	,	0.000715528812223519	)
(	1.2728	,	0.000774438269747822	)
(	1.2869	,	0.000830873041630308	)
(	1.3011	,	0.000884517658533734	)
(	1.3152	,	0.000932865252360704	)
(	1.3294	,	0.000977504958600211	)
(	1.3435	,	0.00101606327946246	)
(	1.3576	,	0.00104773775764358	)
(	1.3718	,	0.00107494868231093	)
(	1.3859	,	0.00109291371837305	)
(	1.4001	,	0.00110462926351808	)
(	1.4142	,	0.0011117	)
};

\addplot[color=green!50!black,line width=0.2mm,solid] coordinates {

 (	0	,	0.00110960000210002	)
(	0.014142	,	0.00110552865499237	)
(	0.028284	,	0.00109421413424795	)
(	0.042426	,	0.00107545325573778	)
(	0.056569	,	0.00104964216264592	)
(	0.070711	,	0.00101717499626416	)
(	0.084853	,	0.000978473052353002	)
(	0.098995	,	0.00093425991338171	)
(	0.11314	,	0.000884951265692072	)
(	0.12728	,	0.000831580236598369	)
(	0.14142	,	0.000774688687211837	)
(	0.15556	,	0.000715357807717229	)
(	0.16971	,	0.000654359740782545	)
(	0.18385	,	0.00059269092990951	)
(	0.19799	,	0.000531285082416211	)
(	0.21213	,	0.00047110593236978	)
(	0.22627	,	0.000413058220111887	)
(	0.24042	,	0.00035816596808463	)
(	0.25456	,	0.000307179668819406	)
(	0.2687	,	0.000261124159931631	)
(	0.28284	,	0.000220673160182656	)
(	0.29698	,	0.000186728207938704	)
(	0.31113	,	0.0001600458385205	)
(	0.32527	,	0.000141362149015428	)
(	0.33941	,	0.000131848495061946	)
(	0.35355	,	0.000132827454892804	)
(	0.3677	,	0.000226587255418305	)
(	0.38184	,	0.000386349863336329	)
(	0.39598	,	0.000607211350519735	)
(	0.41012	,	0.000888554124744239	)
(	0.42426	,	0.00123000154625919	)
(	0.43841	,	0.00162866863766083	)
(	0.45255	,	0.00207807388131414	)
(	0.46669	,	0.0025740236319428	)
(	0.48083	,	0.00310687160375192	)
(	0.49497	,	0.00367108639222778	)
(	0.50912	,	0.00425626077223894	)
(	0.52326	,	0.00485442119412809	)
(	0.5374	,	0.00545608764248706	)
(	0.55154	,	0.00605213192891563	)
(	0.56569	,	0.00663322227886266	)
(	0.57983	,	0.00719082761983626	)
(	0.59397	,	0.00771520614337686	)
(	0.60811	,	0.00819981686998557	)
(	0.62225	,	0.00863531031753926	)
(	0.6364	,	0.00901673166352975	)
(	0.65054	,	0.00933681899685326	)
(	0.66468	,	0.00959110961758336	)
(	0.67882	,	0.00977632178259288	)
(	0.69296	,	0.00988738385696975	)
(	0.70711	,	0.00992630000244417	)
(	0.72125	,	0.00988737725717113	)
(	0.73539	,	0.00977621334777939	)
(	0.74953	,	0.00959089239226465	)
(	0.76368	,	0.00933669905422682	)
(	0.77782	,	0.00901650899161089	)
(	0.79196	,	0.00863497816268808	)
(	0.8061	,	0.00819959290366296	)
(	0.82024	,	0.00771466817951362	)
(	0.83439	,	0.00719060137910592	)
(	0.84853	,	0.00663268368024588	)
(	0.86267	,	0.00605170179377669	)
(	0.87681	,	0.0054556569604494	)
(	0.89095	,	0.00485389571648382	)
(	0.9051	,	0.00425593657626615	)
(	0.91924	,	0.0036705803947741	)
(	0.93338	,	0.00310664645281693	)
(	0.94752	,	0.00257354945139976	)
(	0.96167	,	0.00207786276717208	)
(	0.97581	,	0.00162832902854429	)
(	0.98995	,	0.00122984149730768	)
(	1.0041	,	0.000888417231316458	)
(	1.0182	,	0.000607150019929177	)
(	1.0324	,	0.000386306815627164	)
(	1.0465	,	0.000226710075305444	)
(	1.0607	,	0.000132795492552458	)
(	1.0748	,	0.00013182916043797	)
(	1.0889	,	0.000141342627560867	)
(	1.1031	,	0.000160035448755643	)
(	1.1172	,	0.000186708998029018	)
(	1.1314	,	0.000220664709004408	)
(	1.1455	,	0.000261123976972625	)
(	1.1597	,	0.000307163335162255	)
(	1.1738	,	0.000358176425622904	)
(	1.1879	,	0.000413032222602547	)
(	1.2021	,	0.000471107634836244	)
(	1.2162	,	0.000531268302838594	)
(	1.2304	,	0.000592673693444377	)
(	1.2445	,	0.000654342722776069	)
(	1.2587	,	0.000715320704666795	)
(	1.2728	,	0.000774661899376108	)
(	1.2869	,	0.000831522315365018	)
(	1.3011	,	0.000884904200477091	)
(	1.3152	,	0.000934181513176642	)
(	1.3294	,	0.000978415253014792	)
(	1.3435	,	0.00101707627372189	)
(	1.3576	,	0.00104954321091654	)
(	1.3718	,	0.00107535410853402	)
(	1.3859	,	0.00109411424668542	)
(	1.4001	,	0.00110532891824861	)
(	1.4142	,	0.00110950000000065	)
};
\legend{$L=2$, $L=4$, $L=6$}

	\end{axis}
	\end{tikzpicture}

%% file: large.tex
%
\label{sec:large}
In this section, we study the influence of
the specific operator formulation on the 
performance and scalability 
of an all-at-once multigrid method. Our solver considers the discrete saddle point problem
\begin{align*}
\begin{pmatrix}
A & B^\top \\ B & -C
\end{pmatrix}
\begin{pmatrix}
\bu_h \\ p_h
\end{pmatrix} =
\begin{pmatrix}
\bs{f}_h \\ 0
\end{pmatrix},
\end{align*}
where the submatrices are associated with nodal basis functions and the associated bilinear forms of the equal-order scheme.  Let us recall that only $A$ depends on the specific operator formulation, while $C$ depends on the stabilization that is employed, and $B$ denotes the weak divergence operator.
We use a variable multigrid $V_{{\rm var}}$-cycle with three pre- and post-smoothing steps  of  Uzawa-type, i.e., in the $(k+1)$th-smoothing step we solve
\begin{equation}\label{Uzawa_smooth}
\begin{aligned}
\mathbf{u}_{k+1} &= \mathbf{u}_k + \hat{A}^{-1} (\mathbf{f} - A \mathbf{u}_k - B^\top \mathbf{p}_k),\\
\mathbf{p}_{k+1} &= \mathbf{p}_k + \hat{S}^{-1} (B \mathbf{u}_{k+1} - C \mathbf{p}_k),
\end{aligned}
\end{equation}
where $\hat{A}$ and $\hat{S}$ denote suitable approximations for $A$ and the Schur-complement $S$, respectively; cf. \cite{zulehner_2002,schoeberl-zulehner_2003,gmeiner-huber-john-ruede-wohlmuth_2015}. In particular for $\hat{A}$, we use a symmetric hybrid parallel variant of a row-wise red-black colored Gauss-Seidel method and for $\hat{S}$, we select the forward variant of the Gauss-Seidel method applied to the stabilization matrix $C$, with  $\omega=0.3$ as under-relaxation factor. Standard restriction and prolongation operators define the transfer between the meshes.

Before comparing the formulations in the context of the all-at-once multigrid method, let us comment on the coarse grid solver. As long as a direct coarse grid solver is used, the use of $a_{\tr}(\cdot,\cdot)$ is not problematic. However, quite often, especially in case of moderate viscosity contrasts, preconditioned Krylov subspace solvers may give shorter run-times. Here we apply a block-diagonal preconditioned minimal residual (MINRES) method, where a Jacobi-preconditioned CG-iteration is used for the velocity block and a lumped mass-matrix for the pressure. The stopping criteria are selected such that the multigrid convergence rates do not deteriorate.

However, this preconditioner cannot be straightforwardly applied for $a_{\tr}(\cdot,\cdot)$. Therefore, we exchange the stencil-formulation by the strain-based Stokes formulation on the coarse grid and apply the preconditioned MINRES method there. Although the stencil-formulation is more efficient and faster than the strain-based Stokes formulation when iterating on finer grids, the coarse grid
performance is dominated by communication. Thus, we can choose the more expensive but positive definite strain-based Stokes formulation on the coarse mesh.

 A naive replacement  of  $a_{\tr}(\cdot,\cdot)$
by  $a_{\sym}(\cdot,\cdot)$ on the coarse level will fail.
We recall that $a_{\tr}(\cdot,\cdot)$ is only equivalent to $a_{\sym}(\cdot,\cdot)$ when $\div (\bu) = 0$. On the coarse grid, the right hand side  of the discrete mass conservative equation is typically non-zero. Thus, the two discrete coarse level systems
\begin{align*}
\begin{pmatrix}
A_{\tr} & B^\top \\ B & -C
\end{pmatrix}
\begin{pmatrix}
\bs{z} \\ \bs{q}
\end{pmatrix} =
\begin{pmatrix}
\bs{r}_1 \\ \bs{r}_2
\end{pmatrix}
 \quad \text{and} \quad
\begin{pmatrix}
A_{\sym} & B^\top \\ B & -C
\end{pmatrix}
\begin{pmatrix}
\bs{z} \\ \bs{q}
\end{pmatrix} =
\begin{pmatrix}
\bs{r}_1 \\ \bs{r}_2
\end{pmatrix}
\end{align*} 
do not approximate the same physical system.
Here all matrices are associated with the coarse level, and the subindex in $A$ reflects the selected bilinear form.
To account for this coarse grid inconsistency between the two systems, we modify the coarse grid right-hand side  of the momentum equation. Replacing $\bs{r}_1 $  by $\bs{r}_1 - B^{\top} M^{-1}\bs{r}_2$, where $M$ is the lumped mass matrix, i.e., a diagonal matrix, we find that
\begin{align*}
\begin{pmatrix}
A_{\tr} & B^\top \\ B & -C
\end{pmatrix}
\begin{pmatrix}
\bs{z} \\ \bs{q}
\end{pmatrix} =
\begin{pmatrix}
\bs{r}_1 \\ \bs{r}_2
\end{pmatrix}
 \quad \text{and} \quad 
\begin{pmatrix}
A_{\sym} & B^\top \\ B & -C
\end{pmatrix}
\begin{pmatrix}
\bs{z} \\ \bs{q}
\end{pmatrix} =
\begin{pmatrix}
\bs{r}_1 - B^\top M^{-1} \bs{r}_2 \\ \bs{r}_2
\end{pmatrix}
\end{align*} 
are approximations of the same physical system, and we can use the modified right hand side in combination with $A_{\sym}$ for the coarse mesh solver.

To test the multigrid performance, we use two different architectures and base
our implementation on the massively-parallel geometric multigrid framework HHG~\cite{BBergen_FHuelsemann_2004,bergen-gradl-ruede-huelsemann_2006,CPE:CPE2968,gmeiner-ruede-stengel-waluga-wohlmuth_2015,gmeiner-huber-john-ruede-wohlmuth_2015}.
Firstly, we consider
a two-socket server system equipped with 
Intel Xeon E5-2699 processors, 32 cores in total and 256\,GB main memory.
Due to the matrix-free aproach even on such a low-cost machine, we can solve for
$10^{9}$ DoF. Secondly, the IBM Blue Gene/Q supercomputing system JUQUEEN\footnote{www.fz-juelich.de} with a peak performance of more than 5.9 petaflop/s, currently ranked on position 11 of the TOP500 list (46th edition, November 2015) is used. On this system, Stokes flow problems with $10^{13}$ DoF can be considered \cite{gmeiner-huber-john-ruede-wohlmuth_2015}.
As test example, we reuse the model problem of Section\,\ref{ssub:visc_column}. 

 In Table~\ref{tab:timeANDconv}, we consider the refinement levels $L=5,6,7,8$ on the intel machine and level $L=2$ with 19\,652 DoF for the coarse grid. The asymptotic convergence rate $\rho$ for the residual reduction is approximated by using the convergence rate after 200 iterations with rescaling in each step. For all our settings, the convergence rates are robust with respect to the problem size  and are similar  valued.  Both, the modified stencil and the gradient-based formulations, show almost constant rates $\rho=0.241$ over the levels and are additional slightly better than the ones of the strain-based Stokes operator $\rho =0.255$ on $L=8$.

\begin{table}[h]
 \centering\scriptsize
\begin{tabular}{c||cc|cc|cc}
\hline
&\multicolumn{2}{c|}{strain-based form} & \multicolumn{2}{c|}{modified stencil form} & \multicolumn{2}{c}{gradient-based form}\\
\hline
level	&	    T [s]	&	$\rho$ 	&	T [s]	&	$\rho$ 	&	T [s]	&	$\rho$ \\
\hline
5	&	1.18 		&	0.296	 &	1.12 		&	 0.244	&	1.34  	&  0.240 \\
6	&	3.02 		&	0.280	&	2.45 		&	 0.242	&	2.64 		&  0.240 \\
7	&	13.95 	&	0.263	&	9.04 		&	 0.241	&	9.26 		&  0.241 \\			
8	&	93.74 	&	0.255	&	54.08 	&	 0.241	&	54.22 	&  0.241 \\
\hline
\end{tabular}
\caption{\label{tab:timeANDconv}
Time of one $V_{{\rm var}}(3,3)$-cycle and asymptotic convergence rates for the different formulations on the Intel machine}
\end{table}

 Increasing the refinement level, the run-times for the modified stencil and gradient-based form are up to a factor 1.73 faster than the strain-based Stokes operator. The modified stencil and the strain-based Stokes-operator need a similar number of coarse grid iterations to reach the stopping criteria, while the pure gradient-based operator requires more iterations such that the run-times are slightly higher than for the Stencil operator.

On JUQUEEN, we perform a weak scalability study and report the total time and the time without coarse grid solver (w.c.). For all three formulations, the run-times without coarse grid time stay almost constant in the weak scaling. Including the coarse grid, the run-times increase due to the suboptimal coarse grid solver.  In the weak scaling, the difference between the strain-based Stokes and modified stencil operator is between 7.65s and 8.78s. When neglecting the coarse grid, the difference increases to a factor of around 1.61 between the strain-based Stokes operator and the modified stencil and the gradient-based operator. 

\begin{table}[ht]
 \centering\scriptsize
\begin{tabular}{cc||cc|cc|cc}
 \hline
&&\multicolumn{2}{c|}{strain-based form} & \multicolumn{2}{c|}{modified stencil form} & \multicolumn{2}{c}{gradient-based form}\\
\hline
proc.	& DoF	&	    T [s]	&	T (w.c.) [s]	&T [s]	&	T (w.c.) [s]	 	&	T [s]	&	T (w.c.) [s]\\
\hline
128		&	$5.4\cdot 10^8$		&	25.83 &	21.92	&	18.18 &	14.31	 	& 18.61 	& 14.28	\\
768		&	$4.3\cdot 10^9$		&	32.04 &	21.93	&	23.26 &	13.90	 	& 25.44 	& 13.90	\\
8\,192	&	$3.4\cdot 10^{10}$	       &	39.37 &	21.33	&	31.50 & 	13.69	 	& 36.86 	& 13.27	\\
65\,536	&	$2.8\cdot 10^{11}$	       &	71.09 &	21.41	&	63.23 &	13.27	 	& 69.22 	& 13.26	\\
\hline
\end{tabular}
\caption{\label{tab:timeANDconv2}
Weak scaling of one $V_{{\rm var}}(3,3)$-cycle and asymptotic convergence rates for the different formulations on JUQUEEN. All timings are given in seconds.}
\end{table}

%% file: geo.tex
%
\label{sec:geo}
In our last example, we study the influence of the different formulations on the solution quality for a simplified geophysical setup.
It is widely assumed that there is a huge viscosity contrast
between  the lower Earth's mantle and the asthenosphere, which is a mechanically
weak layer in the uppermost mantle; 
cf., e.g., \cite{haskell_35,mitrovica_96,hartley_11,parnell-turner_14}.
The precise depth of the asthenosphere is still unknown, although it is  accepted that depth and viscosity contrast are closely coupled. 
In order to take the Haskell constraint into account and to study the influence of two different asthenospheric depths $d_a  =660 $\,[km] and $d_a= 410$\,[km], we choose the piecewise constant viscosity
\begin{align}
\nu = 
  \begin{cases} 
    10^{22}\,\mbox{[Pa\,s]} & \text{for } r <  |\mathbf{x}| < R - d_a, \\
     10^{21} \left(\frac{d_a}{1000 [km]}\right)^3\,\mbox{[Pa\,s]} 
        & \text{for }  |\mathbf{x}| \geq R - d_a , \\
\end{cases}
\end{align}
see Weism\"uller~et~al.~\cite{weismueller_15}, resulting in a viscosity jump of $34.78$ and $145.09$, respectively.
Here, $R$ is
the radius of the Earth and $R-r$ is the thickness of the mantle close to $3000$ [km].
The right hand side of the incompressible Stokes system is given by a temperature driven  buoyancy term $(-\rho \mathbf{g})$, i.e. the product of a density field $\rho$~[kg/m$^3$] and
the gravitational acceleration within
the mantle $\mathbf{g}$~[m/s$^2$], which we prescribe as a vector of magnitude $10$~[m/s$^2$] 
pointing towards the center of the Earth. The density $\rho$ is 
obtained from a tomographic model of seismic wave speeds within the Earth 
\cite{grand_97}, converted to densities with the mineralogical model of 
Stixrude~et~al.~\cite{stixrude_05}; see also \cite{bauer-bunge-drzisga-gmeiner-huber-john-mohr-ruede-stengel-waluga-weismueller-wellein-wittmann-wohlmuth_2015}. 
At the outer boundary, Dirichlet conditions for the velocity are specified according to  tectonic plate data \cite{mueller_08}, and at the inner boundary free-slip conditions are assumed. Both types of boundary conditions yield $\bu \cdot \bs{n} =0$. 

For the rather large viscosity contrasts, which are considered here,
our coarse mesh solver consumes a considerable amount of the
computation time. Thus, we shall not compare performance figures here
as we did in the previous section, but only examine the impact on the
solution quality. An investigation of other coarse grid solvers,
combining algebraic and geometric multigrid techniques, can be found for instance
in the recent contribution of Rudi~et~al.~\cite{Rudi:2015:EIS:2807591.2807675}.

We perform our geophysical simulations using the Uzawa-multigrid
$V_{\text{\small var}}$--cycles introduced in Section~\ref{sec:large}
to solve the geophysical problem with the different
formulations. The initial mesh consists of an icosahedral mesh of the
spherical shell with $922,560$ tetrahedrons. This mesh, which resolves
the viscosity jump, is refined twice to obtain the coarse level mesh
and then four more times to build up the multilevel
hierarchy. Altogether, this results in a computational problem consisting of $2.5\cdot 10^9$ degrees of freedom.
The maximal velocities for each formulation and the differences to the strain-based formulation, which we again take as a reference, are reported in Table~\ref{tab:max_velocities}. The difference in the numbers are a first indicator that the gradient-based form yields non-physical results for this setup in the sense that these are not consistent with the strain-based model. In both settings, the maximal difference between our new formulation and the classical strain-based formulation is roughly $1\%$ whereas the gradient-based model results in a maximal difference of more than $10\%$.

\begin{table}[ht]\small\centering
\begin{tabular}{r||r|r|r||r|r}
\hline
  	& \multicolumn{3}{c||}{maximal velocity}	& \multicolumn{2}{c}{maximal difference}\\
\hline
	&	reference		&	mod. stencil	&	grad.-based	&	mod. stencil		&	grad.-based\\
\hline
$d_a = 660$ 	&	17.67	&	17.67	&	19.02	&	0.18		&	3.51\\
$d_a = 410$	&	28.13	&	28.12	&	30.10	&	0.17		&	3.24\\
\hline
\end{tabular}
\caption{\label{tab:max_velocities} Maximal velocity magnitudes (in [m/s]) for the different formulations and maximal absolute differences (in [m/s]) for asthenosphere depths $d_a = 660$ (left) and $d_a = 410$ .}
\end{table}

In the top row of Figure~\ref{fig:asthenosphere_study}, we present simulation
results for the strain-based operator and asthenosphere depths $d_a = 660$ (left) and $d_a = 410$ (right). We scale the color bar  to a maximum velocity of 15~[m/s] and observe that the highest velocities occur in the thin asthenospheric layer. The smaller $d_a$, the higher the maximal velocity is. For both $d_a$, a similar structure in the velocity distribution is obtained.  To get a better feeling for the different formulations and the effect of the physical inconsistent interface coupling in case of the gradient-based formulation, we also provide a scaled
relative point-wise error plot
\begin{equation}\label{eq:rel_error}
\text{Error}_k(x) = \frac{\|\bu_{\text{e}}(x)-\bu_{k}(x)\|}{\max(\|\bu_{\text{e}}(x)\|,5 \cdot 10^{-2}\|\bu_{\text{e}}\|_{L^{\infty}(\Omega)})}, 
\end{equation}
with $k \in \{\text{g},\text{s}\}$, denoting the gradient-based (g)
and the stencil-based (s) discrete solution, respectively. The coloring of  Error$_\text{g}$ is scaled to  50\% and of Error$_\text{s}$ to 2.5\%.
In the center row, Error$_\text{s}$ is shown, and  Error$_\text{g}$ is depicted in the third row.
For both $d_a$, the gradient-based form shows high differences triggered by the incorrect physical formulation at the interface and the inner boundary. We observe that this is  a global effect and pollutes into the interior of the lower mantle.
However, for the new stencil based formulation no inconsistency error enters at the interface, and thus the relative difference is quite small. For the larger asthenospheric depth, there is almost no visible difference while for the smaller layer moderate differences can be observed in regions where we have small velocities.

\begin{figure}[ht]
\begin{center}
\includegraphics[width=.4\textwidth]{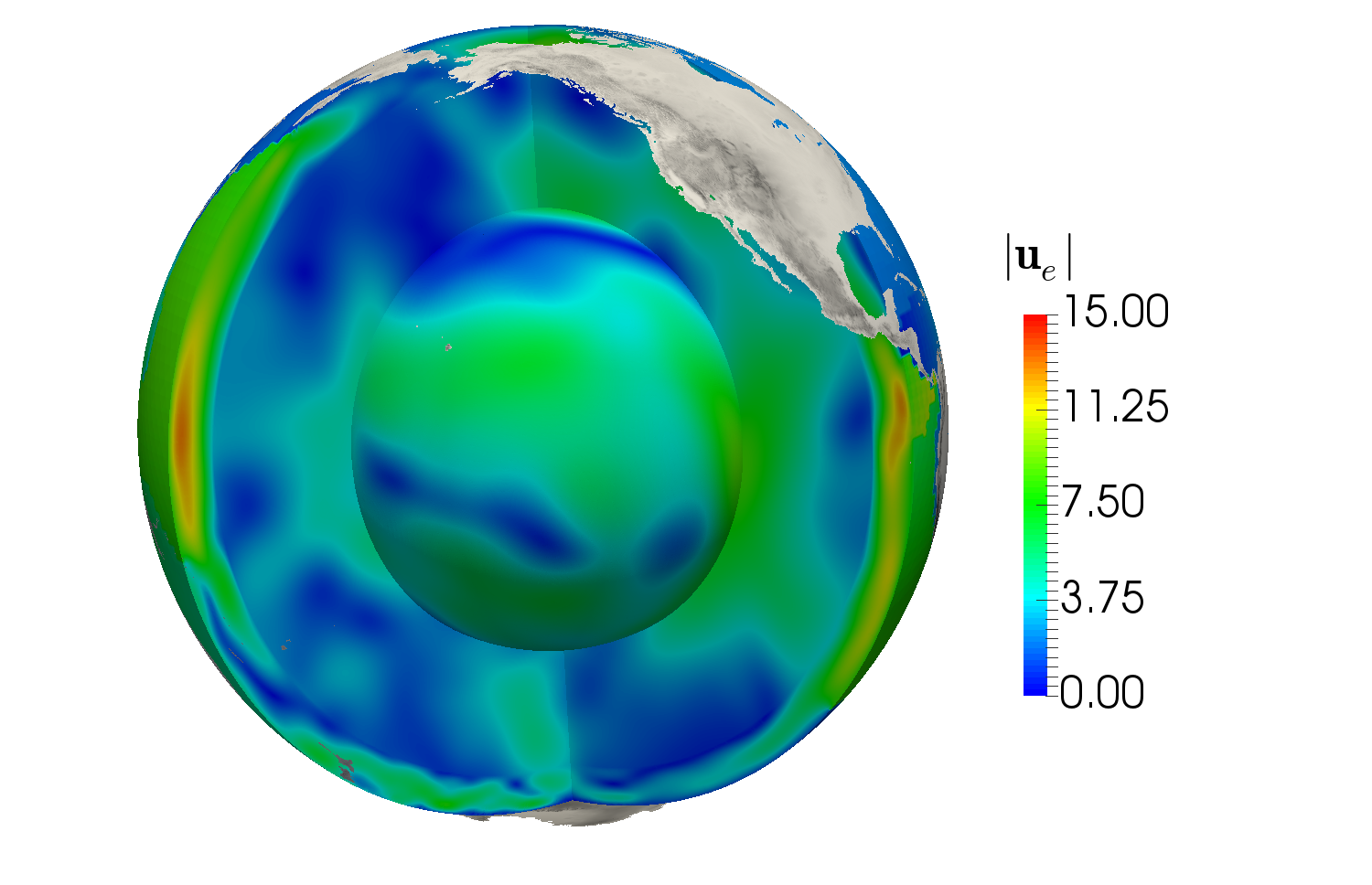}
\includegraphics[width=.4\textwidth]{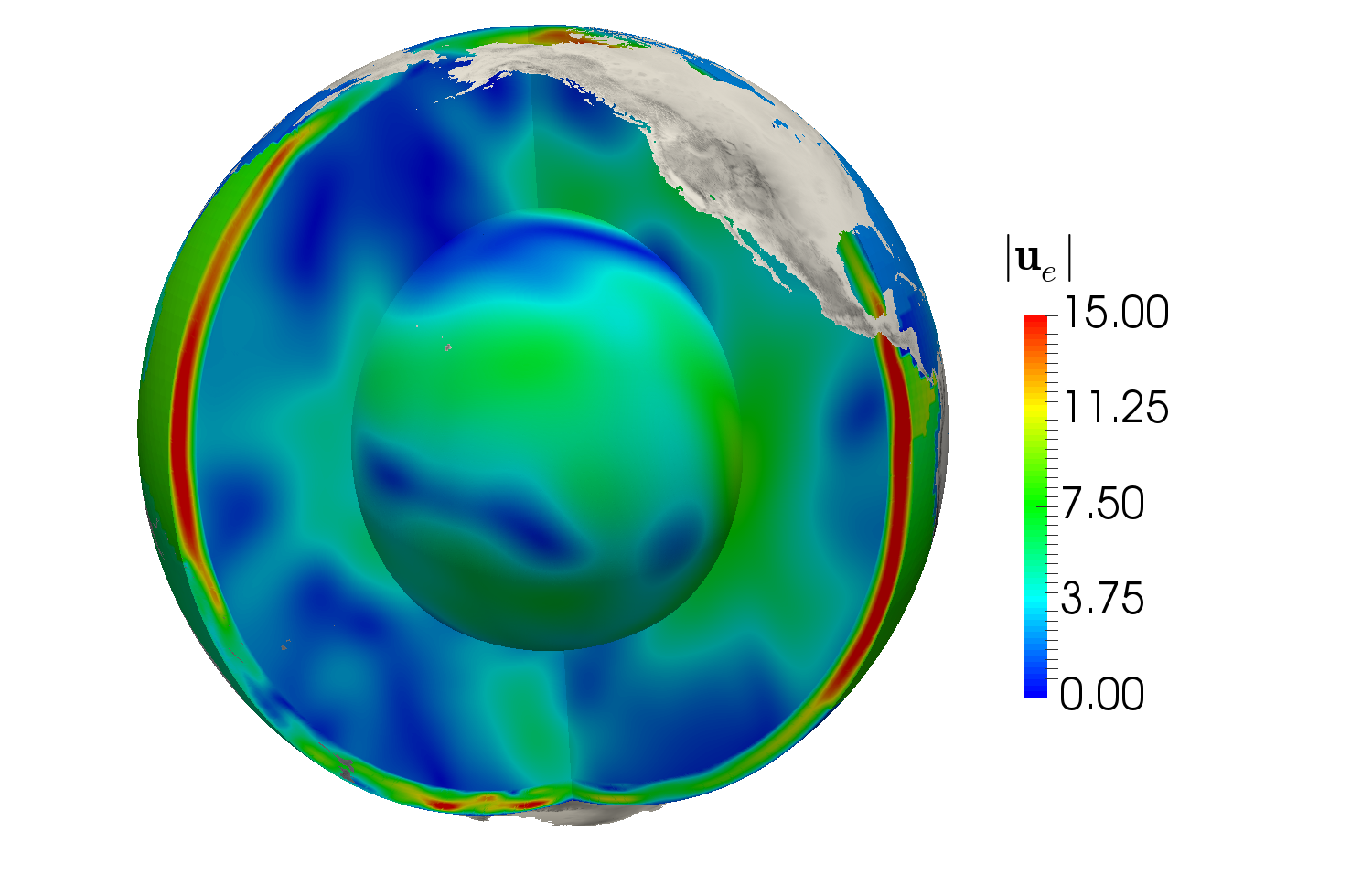}
\\
\includegraphics[width=.4\textwidth]{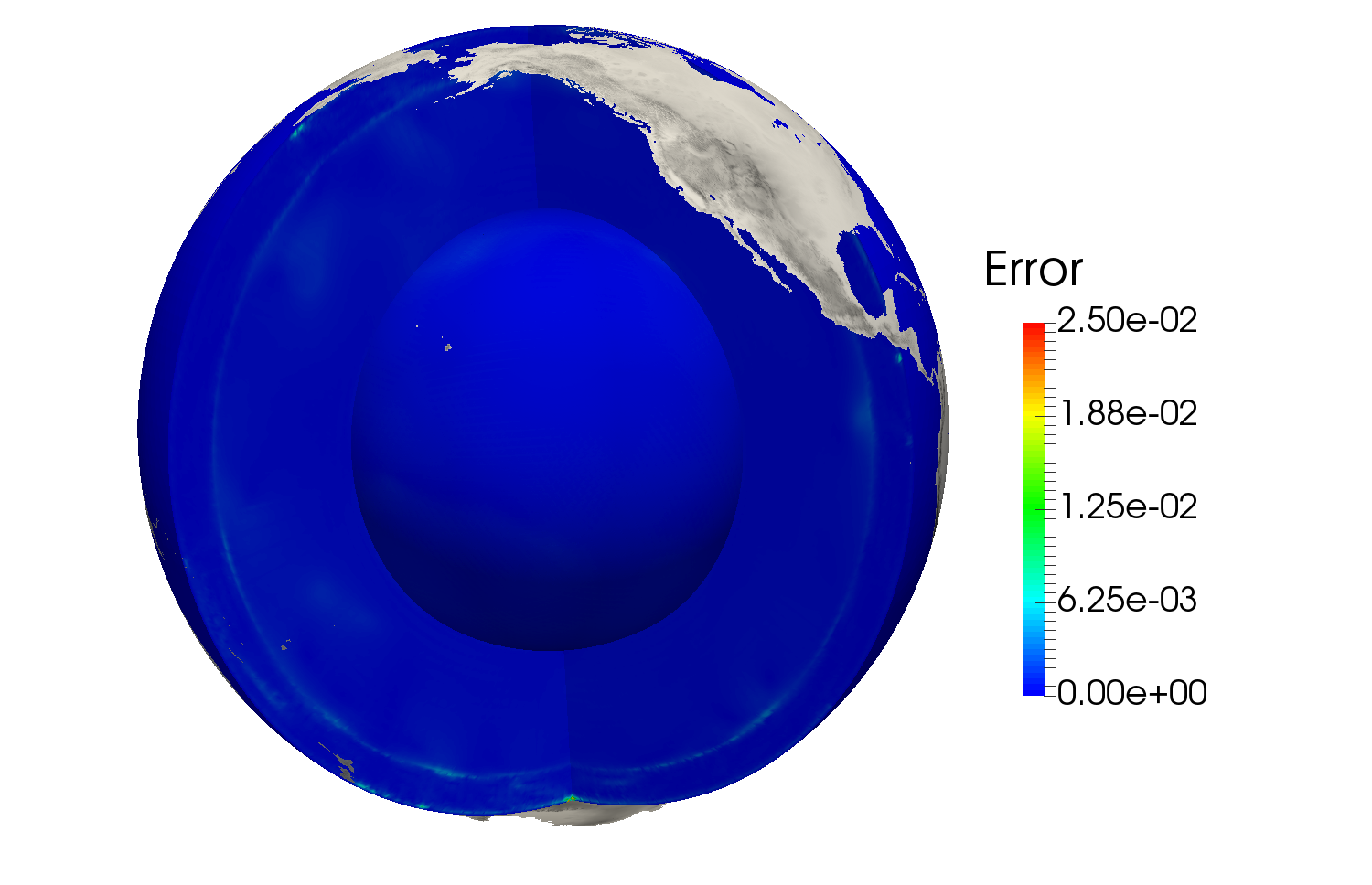}
\includegraphics[width=.4\textwidth]{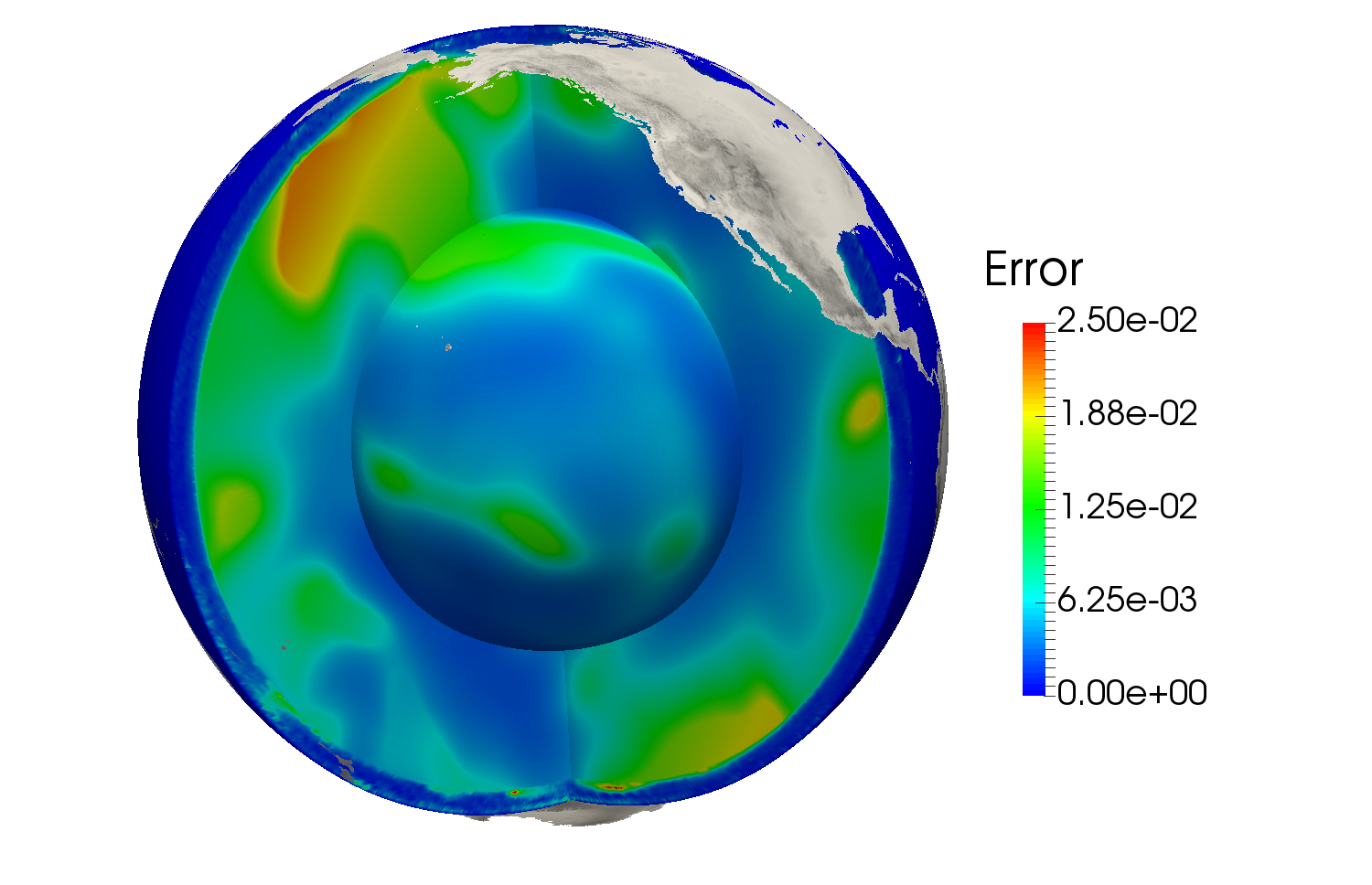}\\
\includegraphics[width=.4\textwidth]{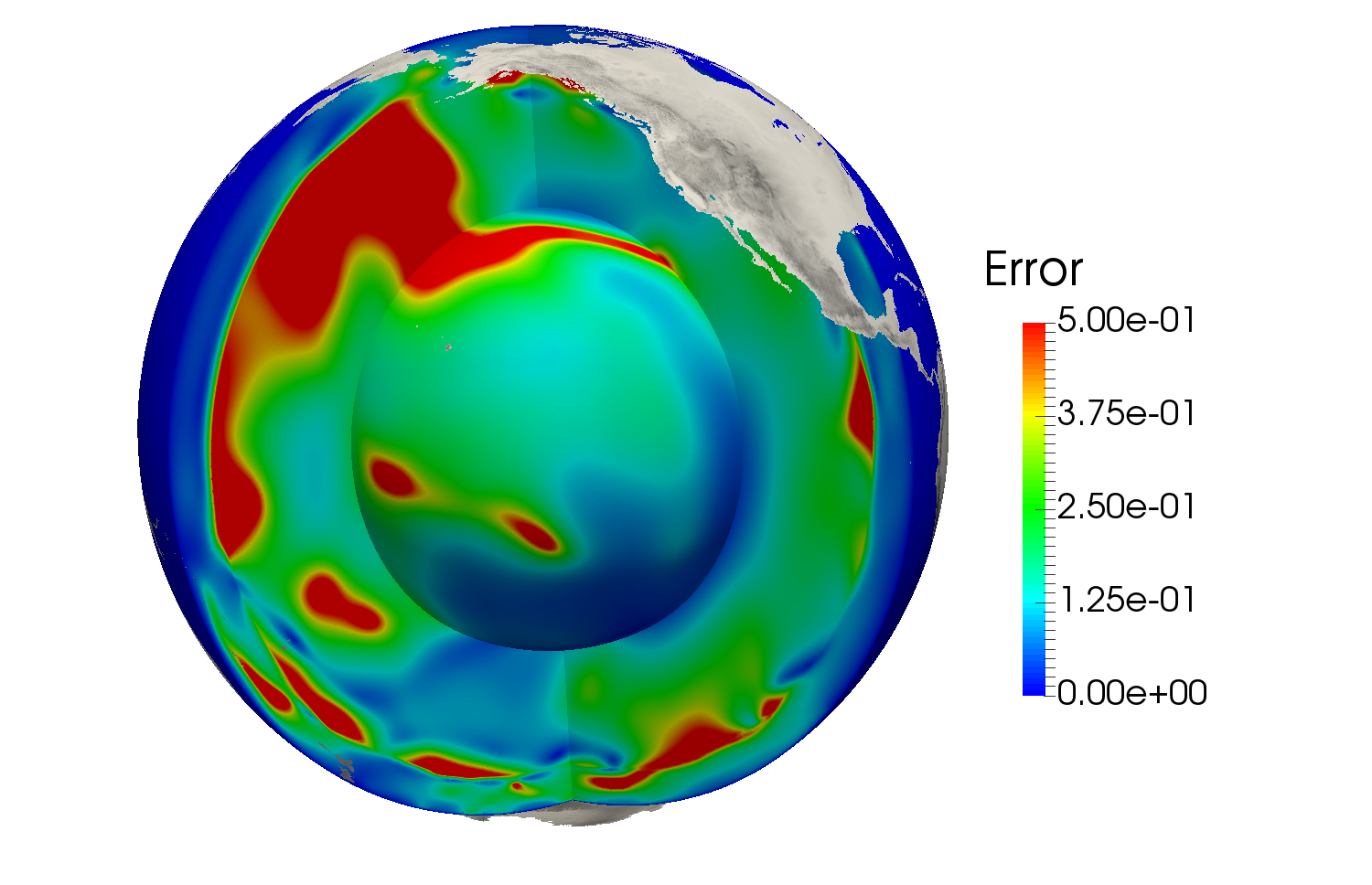}
\includegraphics[width=.4\textwidth]{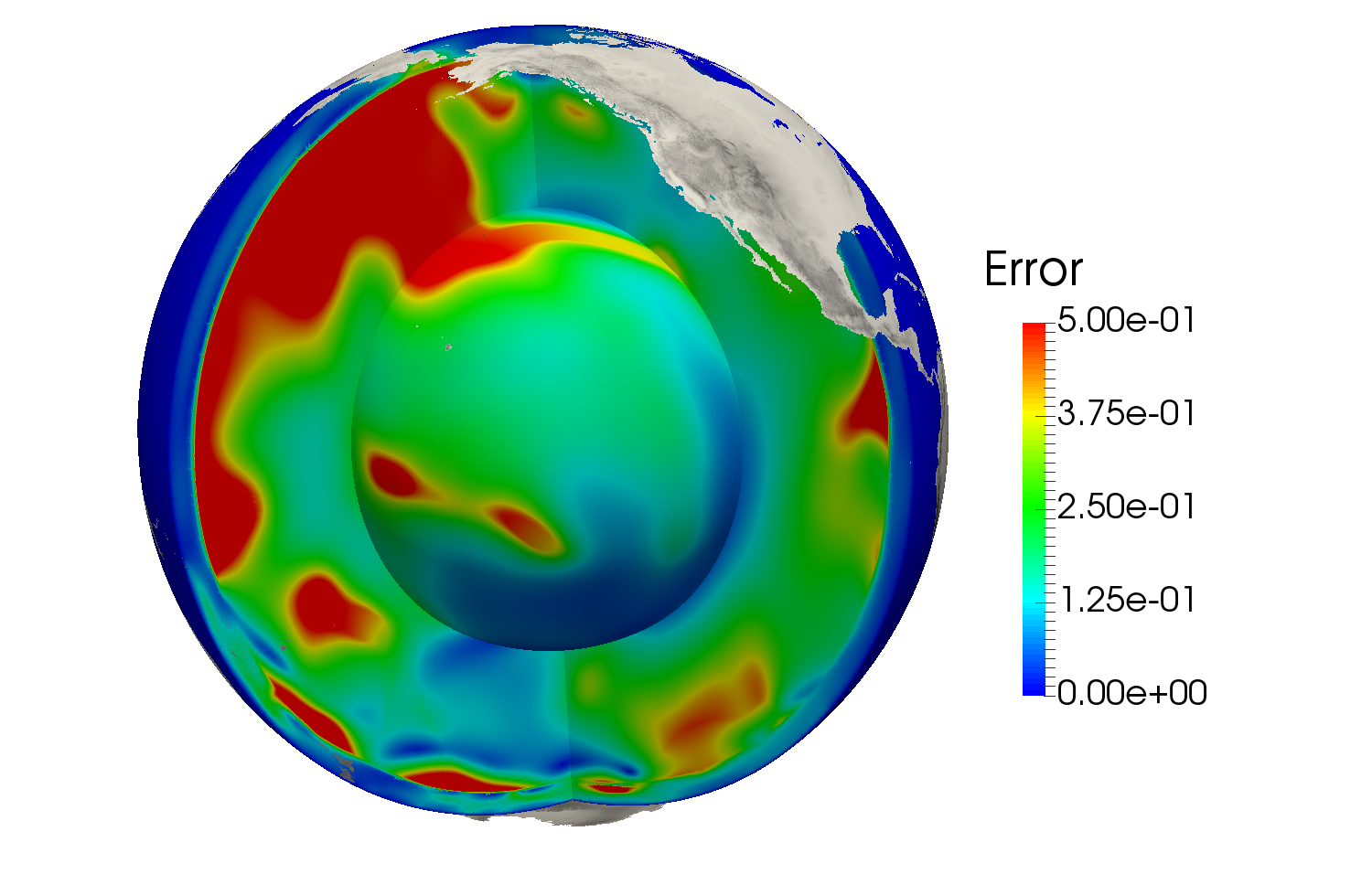}
\end{center}
\caption{\label{fig:asthenosphere_study}
Study of different asthenosphere depths $d_a = 660$ (left) and $d_a = 410$ (right): Velocity-magnitude of the strain-based Stokes formulation \eqref{eq:standard} (top). Relative error for the new stencil modification \eqref{eq:newbil}~(middle) and the gradient-based formulation~(bottom). 
}
\end{figure}

%% file: 2016_Viscosityjumps.bbl
\begin{thebibliography}{10}

\bibitem{bauer-bunge-drzisga-gmeiner-huber-john-mohr-ruede-stengel-waluga-weismueller-wellein-wittmann-wohlmuth_2015}
S.~Bauer, H.-P. Bunge, D.~Drzisga, B.~Gmeiner, M.~Huber, L.~John, M.~Mohr,
  U.~R{\"u}de, H.~Stengel, C.~Waluga, J.~Weism{\"u}ller, G.~Wellein,
  M.~Wittmann, and B.~Wohlmuth.
\newblock {Solution Techniques for the Stokes System: A priori and a posteriori
  modifications, resilient algorithms}.
\newblock In {\em SPPEXA proceedings}, 2015.
\newblock accepted.

\bibitem{bergen-gradl-ruede-huelsemann_2006}
B.~Bergen, T.~Gradl, U.~R{\"u}de, and F.~H{\"u}lsemann.
\newblock A massively parallel multigrid method for finite elements.
\newblock {\em Computing in Science and Engineering}, 8(6):56--62, 2006.

\bibitem{BBergen_FHuelsemann_2004}
B.~Bergen and F.~H{\"u}lsemann.
\newblock {Hierarchical hybrid grids: data structures and core algorithms for
  multigrid}.
\newblock {\em Numerical Linear Algebra with Applications}, 11:279--291, 2004.

\bibitem{Bey:1995fk}
J.~Bey.
\newblock Tetrahedral grid refinement.
\newblock {\em Computing}, 55(4):355--378, 1995.

\bibitem{boffi2013mixed}
D.~Boffi, F.~Brezzi, and M.~Fortin.
\newblock {\em Mixed finite element methods and applications}.
\newblock Springer, 2013.

\bibitem{Boffi:2012aa}
D.~Boffi, N.~Cavallini, F.~Gardini, and L.~Gastaldi.
\newblock Local mass conservation of {S}tokes finite elements.
\newblock {\em Journal of Scientific Computing}, 52(2):383--400, 2012.

\bibitem{brezzi1984stabilization}
F.~Brezzi and J.~Pitk\"aranta.
\newblock On the stabilization of finite element approximations of the {S}tokes
  equations.
\newblock In W.~Hackbusch, editor, {\em Efficient Solutions of Elliptic
  Systems}. Springer, 1984.

\bibitem{CouBorLeLinRebWan:2013}
B.~R. Cousins, S.~{Le Borne}, A.~Linke, L.~G. Rebholz, and Z.~Wang.
\newblock Efficient linear solvers for incompressible flow simulations using
  {S}cott--{V}ogelius finite elements.
\newblock {\em Numerical Methods for Partial Differential Equations},
  29(4):1217--1237, 2013.

\bibitem{gerya2009introduction}
T.~Gerya.
\newblock {\em Introduction to Numerical Geodynamic Modelling}.
\newblock Cambridge University Press, 2009.

\bibitem{gmeiner-huber-john-ruede-wohlmuth_2015}
B.~Gmeiner, M.~Huber, L.~John, U.~R{\"u}de, and B.~Wohlmuth.
\newblock A quantitative performance analysis for {Stokes} solvers at the
  extreme scale, 2015.
\newblock submitted, arXiv:1511.02134.

\bibitem{CPE:CPE2968}
B.~Gmeiner, H.~K{\"o}stler, M.~St{\"u}rmer, and U.~R{\"u}de.
\newblock Parallel multigrid on hierarchical hybrid grids: a performance study
  on current high performance computing clusters.
\newblock {\em Concurrency and Computation: Practice and Experience}, 2012.

\bibitem{gmeiner-ruede-stengel-waluga-wohlmuth_2015}
B.~Gmeiner, U.~R{\"u}de, H.~Stengel, C.~Waluga, and B.~Wohlmuth.
\newblock Performance and {S}calability of {H}ierarchical {H}ybrid {M}ultigrid
  {S}olvers for {S}tokes {S}ystems.
\newblock {\em SIAM J. Sci. Comput.}, 37(2):C143--C168, 2015.

\bibitem{gmeiner-ruede-stengel-waluga-wohlmuth_2015_2}
B.~Gmeiner, U.~R{\"u}de, H.~Stengel, C.~Waluga, and B.~Wohlmuth.
\newblock Towards textbook efficiency for parallel multigrid.
\newblock {\em Numer. Math. Theory Methods Appl.}, 8, 2015.

\bibitem{GWW14}
B.~Gmeiner, C.~Waluga, and B.~Wohlmuth.
\newblock Local mass-corrections for continuous pressure approximations of
  incompressible flow.
\newblock {\em SIAM Journal on Numerical Analysis}, 52(6):2931--2956, 2014.

\bibitem{grand_97}
S.~P. Grand, R.~D. van~der Hilst, and S.~Widiyantoro.
\newblock Global seismic tomography: A snapshot of convection in the earth.
\newblock {\em GSA Today}, 7:1--7, 1997.

\bibitem{gross2011numerical}
S.~Gross and A.~Reusken.
\newblock {\em Numerical methods for two-phase incompressible flows},
  volume~40.
\newblock Springer Science \& Business Media, 2011.

\bibitem{HC89}
B.~Hager and R.~Clayton.
\newblock {\em Mantle Convection}, chapter Constraints on the structure of
  mantle convection using seismic observations, flow models, and the geoid,
  pages 657--764.
\newblock Gordon and and Breach, New York, 1989.

\bibitem{HanLarZah:2012:preprint}
P.~{Hansbo}, M.~G. {Larson}, and S.~{Zahedi}.
\newblock {A cut finite element method for a Stokes interface problem}.
\newblock {\em ArXiv e-prints}, May 2012.

\bibitem{hartley_11}
R.~Hartley, G.~Roberts, N.~White, and C.~Richardson.
\newblock Transient convective uplift of an ancient buried landscape.
\newblock {\em Nature Geosci.}, 4:562--565, 2011.

\bibitem{haskell_35}
N.~A. Haskell.
\newblock The motion of a fluid under a surface load.
\newblock {\em Physics}, 6:265--269, 1935.

\bibitem{hughes1987new}
T.~J.~R. Hughes and L.~P. Franca.
\newblock A new finite element formulation for computational fluid dynamics:
  {VII}. the {S}tokes problem with various well-posed boundary conditions:
  symmetric formulations that converge for all velocity/pressure spaces.
\newblock {\em Computer Methods in Applied Mechanics and Engineering},
  65(1):85--96, 1987.

\bibitem{ito2006interface}
K.~Ito and Z.~Li.
\newblock Interface conditions for stokes equations with a discontinuous
  viscosity and surface sources.
\newblock {\em Applied mathematics letters}, 19(3):229--234, 2006.

\bibitem{LimIdeRosOna:2007}
A.~Limache, S.~Idelsohn, R.~Rossi, and E.~O{\~n}ate.
\newblock The violation of objectivity in {L}aplace formulations of the
  {N}avier--{S}tokes equations.
\newblock {\em International journal for numerical methods in fluids},
  54(6-8):639--664, 2007.

\bibitem{LoggWellsEtAl2012a}
A.~Logg, K.-A. Mardal, and G.~N. Wells.
\newblock {\em DOLFIN: a C++/Python Finite Element Library}, volume~84 of {\em
  Lecture Notes in Comp. Science and Engineering}, chapter~10.
\newblock Springer, 2012.

\bibitem{mitrovica_96}
J.~X. Mitrovica.
\newblock Haskell [1935] revisited.
\newblock {\em J. Geophys. Res.}, 101:555--569, 1996.

\bibitem{mueller_08}
R.~D. M{\"u}ller, M.~Sdrolias, C.~Gaina, and W.~R. Roest.
\newblock Age, spreading rates, and spreading asymmetry of the world's ocean
  crust.
\newblock {\em Geochem. Geophy. Geosy.}, 9:1525--2027, 2008.

\bibitem{neff2015poincare}
P.~Neff, D.~Pauly, and K.-J. Witsch.
\newblock {P}oincar\'e meets {K}orn via {M}axwell: Extending {K}orn's first
  inequality to incompatible tensor fields.
\newblock {\em Journal of Differential Equations}, 258(4):1267 -- 1302, 2015.

\bibitem{neilan2015discrete}
M.~Neilan.
\newblock Discrete and conforming smooth de {R}ham complexes in three
  dimensions.
\newblock {\em Mathematics of Computation}, 2015.

\bibitem{olshanskii2006analysis}
M.~A. Olshanskii and A.~Reusken.
\newblock Analysis of a {S}tokes interface problem.
\newblock {\em Numerische Mathematik}, 103(1):129--149, 2006.

\bibitem{parnell-turner_14}
R.~Parnell-Turner, N.~White, T.~Henstock, B.~Murton, J.~Maclennan, and S.~M.
  Jones.
\newblock A continuous 55 million year record of transient mantle plume
  activity beneath {I}celand.
\newblock {\em Nature Geosci.}, 7:914--919, 2014.

\bibitem{pironneau1989finite}
O.~Pironneau.
\newblock {\em Finite element methods for fluids}.
\newblock Wiley Chichester, 1989.

\bibitem{Pom11}
W.~Pompe.
\newblock Counterexamples to {K}orn's inequality with non-constant rotation
  coefficients.
\newblock {\em Math. Mech. Solids}, 16:172--176, 2011.

\bibitem{QuaVal:1999}
A.~Quarteroni and A.~Valli.
\newblock {\em Domain Decomposition Methods for Partial Differential
  Equations}.
\newblock Numerical mathematics and scientific computation. Clarendon Press,
  1999.

\bibitem{Rudi:2015:EIS:2807591.2807675}
J.~Rudi, A.~Malossi, T.~Isaac, G.~Stadler, M.~Gurnis, P.~Staar, Y.~Ineichen,
  C.~Bekas, A.~Curioni, and O.~Ghattas.
\newblock An extreme-scale implicit solver for complex {PDE}s: Highly
  heterogeneous flow in {E}arth's mantle.
\newblock In {\em Proceedings of the International Conference for High
  Performance Computing, Networking, Storage and Analysis}, pages 5:1--5:12,
  2015.

\bibitem{schoeberl-zulehner_2003}
J.~Sch{\"o}berl and W.~Zulehner.
\newblock On {S}chwarz-type smoothers for saddle point problems.
\newblock {\em Numer. Math.}, 95(2):377--399, 2003.

\bibitem{ScoZha:90}
L.~Scott and S.~Zhang.
\newblock Finite element interpolation of nonsmooth functions satisfying
  boundary conditions.
\newblock {\em Mathematics of Computation}, 54(190):483--493, 1990.

\bibitem{scott1985norm}
L.~R. Scott and M.~Vogelius.
\newblock Norm estimates for a maximal right inverse of the divergence operator
  in spaces of piecewise polynomials.
\newblock {\em Mod{\'e}lisation math{\'e}matique et analyse num{\'e}rique},
  19(1):111--143, 1985.

\bibitem{stixrude_05}
L.~Stixrude and C.~Lithgow-Bertelloni.
\newblock Thermodynamics of mantle minerals -- {I}. {P}hysical properties.
\newblock {\em Geophys. J. Int.}, 162:610--632, 2005.

\bibitem{weismueller_15}
J.~Weism{\"u}ller, B.~Gmeiner, S.~Ghelichkhan, M.~Huber, L.~John, B.~Wohlmuth,
  U.~R{\"u}de, and H.~P. Bunge.
\newblock Fast asthenosphere motion in high-resolution global mantle flow
  models.
\newblock {\em Geophys. Res. Lett.}, 42(18):7429--7435, 2015.

\bibitem{zhang2005new}
S.~Zhang.
\newblock A new family of stable mixed finite elements for the 3d {S}tokes
  equations.
\newblock {\em Mathematics of Computation}, 74(250):543--554, 2005.

\bibitem{zulehner_2002}
W.~Zulehner.
\newblock Analysis of iterative methods for saddle point problems: a unified
  approach.
\newblock {\em Math. Comp.}, 71(238):479--505, 2002.

\end{thebibliography}
